\documentclass[11pt]{amsart}
\usepackage{amssymb,amsmath}
\usepackage[pagewise]{lineno}

\newif\ifextended
\newif\ifdraft

\newcounter{changeno}
\usepackage{hyperref} 

\usepackage{amsthm}
\usepackage{cleveref}

\newif\iffigures
\figurestrue

\usepackage{tikz, tikz-cd}
\usetikzlibrary{arrows, decorations.pathreplacing,calc,intersections,through}
\usepackage{mathtools}
\usepackage{enumerate}
\usepackage{cancel}
\usepackage[normalem]{ulem}
\usepackage{wasysym}
\usepackage{dsfont}

\newcommand{\nc}{\newcommand} 

\nc{\tri}{\triangle}

\nc{\RP}{\R\mathbb P}

\DeclareMathOperator{\length}{length}
\nc{\TSM}{T\hspace{.1em}T^1M}
\nc{\VSM}{V\hspace{.1em}T^1M}

\nc{\reg}{_{\text{reg}}}

\nc{\flatas}{_{\text{flat}\ast}}
\nc{\sing}{_{\text{sing}}}
\nc{\hex}{\hexagon}

\nc{\CC}{\mathcal C}
\nc{\B}{\mathcal B}
\nc{\M}{\mathcal M}

\renewcommand{\P}{\mathbb P}
\newcommand{\cl}[1]{\overline{#1}} 
\newcommand{\then}{\Longrightarrow} 
\newcommand{\wt}[1]{\widetilde{#1}} 
\newcommand{\tl}[1]{\tilde{#1}} 

\newcommand{\Gam}{\Gamma}
\newcommand{\gam}{\gamma}
\newcommand{\Om}{\Omega}

\newcommand{\alp}{\alpha}
\newcommand{\ep}{\epsilon}

\renewcommand{\phi}{\varphi}

\DeclareMathOperator{\SL}{SL} 
\DeclareMathOperator{\PSL}{PSL} 

\DeclareMathOperator{\Stab}{Stab} 
\DeclareMathOperator{\Aut}{Aut}
\DeclareMathOperator{\CAT}{CAT}

\DeclareMathOperator{\bd}{\partial} 
\newcommand{\bdv}[1][]{\partial_{{ #1}}} 
\DeclareMathOperator{\diam}{diam} 
\DeclareMathOperator{\vol}{vol} 

\DeclareMathOperator{\supp}{supp} 

\newcommand{\field}[1]{\mathbb{#1}} 

\newcommand{\R}{\field{R}} 
\newcommand{\Z}{\field{Z}} 
\newcommand{\N}{\field{N}} 

\tikzset{%
  add/.style args={#1 and #2}{to path={%
    ($(\tikztostart)!-#1!(\tikztotarget)$)--($(\tikztotarget)!-#2!(\tikztostart)$)%
  \tikztonodes}}
}

\numberwithin{equation}{section}
\numberwithin{figure}{section}
\numberwithin{table}{section}

\theoremstyle{plain}
\newtheorem{thm}{Theorem}
\numberwithin{thm}{section}
\newtheorem{theorem}[thm]{Theorem}

\newtheorem{proposition}[thm]{Proposition}

\newtheorem{corollary}[thm]{Corollary}

\newtheorem{lemma}[thm]{Lemma}

\newcounter{intro}

\theoremstyle{definition}

\newtheorem{definition}[thm]{Definition}

\newtheorem{remark}[thm]{Remark}

\ifextended
\title[Bowen-Margulis measure for Benoist 3-manifolds: Ext. Version]{Ergodicity of Bowen-Margulis measure for 
some nonstrictly convex Hilbert geometries: Extended version}
\else
\title[Ergodicity of Bowen-Margulis measure for Benoist 3-manifolds]{Ergodicity of Bowen-Margulis measure for 
the Benoist 3-manifolds}
\fi
\author{Harrison Bray}

\begin{document}

\maketitle

\begin{abstract}
  We study the geodesic flow of a class of 3-manifolds introduced by Benoist which have some
  hyperbolicity but are non-Riemannian, not $\CAT(0)$, and with non-$C^1$ geodesic flow. The
  geometries are nonstrictly convex Hilbert geometries in dimension three which admit compact
  quotient manifolds by discrete groups of projective transformations.  We prove the
  Patterson-Sullivan density is canonical, with applications to counting, and construct
  explicitly the Bowen-Margulis measure of maximal entropy. The main result of this work is
  ergodicity of the Bowen-Margulis measure. 
\end{abstract}

\section{Introduction}

In 2004, Yves Benoist released the first results on geodesic flows of
compact quotients of properly convex domains in real projective space
endowed with the Hilbert metric,
proving that strict convexity of the domain is equivalent to an Anosov
geodesic flow of the quotient \cite{Ben1}. Not long after, Benoist produced
nontrivial examples of nonstrictly convex domains with compact quotients in
dimension three, and proved rigid geometric properties for these domains
(\cite{Ben4}, see Theorem \ref{thm:ben3mfldgeom}). 

This family of 3-manifolds, whose quotients we call the Benoist 3-manifolds, lack the Anosov
property but have have similar topological properties to nonpositively curved manifolds which are
rank one. Hence they are promising candidates for studying the geodesic flow.  However, the geometry
is only Finsler and not Riemannian, meaning angles are not defined, the natural metric is not
$\CAT(0)$, and the geodesic flow is not $C^1$. In this work, we extend the approach of Knieper for
rank one manifolds \cite{Kn97,Kn98} and study the geodesic flow of properly convex domains in real
projective space, known also as Hilbert geometries, without the strictly convex hypothesis for the
first time.  We prove the following central result: 

\begin{theorem}
  The Bowen-Margulis measure is an ergodic measure of maximal entropy for geodesic flows of the
  Benoist 3-manifolds. 
  \label{thm:mainthm1}
\end{theorem}

In seeking the main result, we develop the asymptotic geometry and Patterson-Sullivan measures at
infinity.  Let $\delta_\Gam$ denote the critical exponent of the
fundamental group $\Gam$ acting on the universal cover. It is known that
$\delta_\Gam$ is positive in this setting \cite{braytop}. 

\begin{theorem}
  The universal cover of a Benoist 3-manifold admits a 
  Busemann density of
  dimension $\delta_\Gam$ called the Patterson-Sullivan density, and
  Busemann densities of the same dimension $\delta>0$ are unique up to
  constant. 
  \label{thm:mainthm3}
\end{theorem}

Let $S_\Om(x,t)$ be the sphere of Hilbert radius $t$ about $x$ and let
$\vol$ be a natural volume on this sphere. 
As in \cite{sullivan79}, Theorem \ref{thm:mainthm3} can be applied to prove: 

\begin{theorem}
  Let $\Om$ be a properly convex, indecomposable Hilbert geometry of dimension three which admits a
  cocompact action by a 
  discrete, torsion-free group $\Gam$ of projective transformations. 
  Then for all $x\in\Om$, there is a constant $a(x)>0$ such that 
  \[
    \frac1a \leq \frac{\vol S_\Om(x,t)}{e^{\delta_\Gam t}} \leq a.
  \]
  \label{thm:mainthm4}
\end{theorem}

A corollary of Theorem \ref{thm:mainthm4} is  that the group $\Gam$ is divergent. 

\subsubsection*{Historical remarks} 

The Hilbert metric on a properly convex domain in real projective space is
named after Hilbert's proposed solution to his fourth problem; these
geometries
are examples of affine metric spaces for which lines are always geodesic.
Though much work has been done for geodesic flows of strictly convex
Hilbert geometries, little is known on the dynamics in the nonstrictly
convex case \cite{Ben1,CrThese, Cr09, Cr14, CrMar14}.  As alluded to in the
introduction, Benoist first proved that for any properly
convex domain $\Om$ in $n$-dimensional real projective space $\R\P^n$ which
is {\em divisible},
meaning $\Om$ admits a discrete, cocompact action by a group $\Gam$ of
projective transformations, the following are equivalent: (i) $\Om$ is
strictly convex, (ii) the topological boundary $\bd\Om$ is $C^1$, (iii)
$\Gam$ is $\delta$-hyperbolic, and (iv) the geodesic flow of $M=\Om/\Gam$
is Anosov \cite[Theorem 1.1]{Ben1}.  Since the geodesic flow is also
topologically transitive (in fact, mixing, \cite[Theorem 1.2]{Ben1}), it
follows that there is a unique measure of maximal entropy in the strictly
convex setting \cite{bowen75, franco}. 

Benoist then constructed examples of nonstrictly convex, divisible Hilbert
geometries in dimension three which have some hyperbolicity but have
isometrically embedded flats \cite[Proposition 4.2]{Ben4}.  These flats
appear as {\em properly embedded triangles} $\tri$ in $\Om$, meaning
$\tri\subset \Om$ and $\bd\tri\subset \bd\Om$, which are isometric to
$\R^2$ with the hexagonal norm in the Hilbert metric \cite{delaHarpe}.
Moreover, he showed that \emph{any} nonstrictly convex, indecomposable,
divisible properly convex set in $\mathbb R \mathbb P^3$ must have the same structure: 

\begin{theorem}[{\cite[Theorem 1.1]{Ben4}}]
  Let $\Gam < \SL(4,\R)$ be a discrete torsion-free subgroup which divides
  an open, properly convex, indecomposable $\Om\subset \RP^3$, and let $M =
  \Om/\Gam$. Let $\mathcal T$ denote the collection of properly embedded
  triangles in $\Om$. Then 
  \begin{enumerate}[(a)]
    \item Every subgroup in $\Gam$ isomorphic to $\Z^2$ stabilizes a unique triangle $\tri \in \mathcal T$. 
    \item If $\tri_1, \tri_2\in\mathcal T$ are distinct, then $\cl{\tri_1}\cap\cl{\tri_2} = \varnothing$. 
    \item For every $\tri\in \mathcal T$, the stabilizer $\Stab_\Gam(\tri)$ contains an index-two
      $\Z^2$ subgroup. 
    \item The group $\Gam$ has only finitely many orbits in $\mathcal T$. 
    \item The image in $M$ of triangles in $\mathcal T$ is a finite collection $\Sigma$ of disjoint
      tori and Klein bottles, denoted by $T$. If one cuts $M$ along each $T\in \Sigma$, each of the
      resulting connected components is atoroidal.  
    \item Every open line segment in $\bd\Om$ is included in the boundary of some $\tri\in\mathcal
      T$. 
    \item If $\Om$ is not strictly convex, then the set of vertices of triangles in $\mathcal T$ is
      dense in $\bd\Om$.  
  \end{enumerate}
  \label{thm:ben3mfldgeom}
\end{theorem}

This structure is essential to make the arguments needed, and we will refer
back to parts of this theorem throughout the paper. Since a version of this
theorem does not yet exist in higher dimensions, our arguments are valid
only in dimension three.  We call compact quotient manifolds of nonstrictly
convex, properly convex, indecomposable domains the {\em Benoist
3-manifolds}.

The existing theory does not apply to studying the geodesic flows of the
Benoist 3-manifolds.  The geodesic flow of $\Om/\Gam$ has the same
regularity of $\bd\Om$, hence by Benoist's dichotomy, in the nonstrictly
convex setting the geodesic flow is not $C^1$. Crampon's Lyapunov exponents
cannot be computed \cite{Cr14}, and Pesin theory, which requires the flow
to be $C^{1+\alpha}$, does not apply. Knieper's work uses the existence of
an inner product and a notion of angle, so his work on Riemannian rank one
manifolds cannot be directly applied \cite{Kn97, Kn98}. The geometry is not
$\CAT(0)$ because the isometrically embedded flats, which are properly
embedded projective triangles, are not $\CAT(0)$ so we cannot use results
from the thesis of Ricks \cite{delaHarpe, ricks}.

Nonetheless, we can adapt the methods of Knieper in rank one following the
Patterson-Sullivan approach \cite{patterson, sullivan79}.  The irregularity
of the geometry and our techniques to manage this comprise a significant
portion of the paper.  The Bowen-Margulis measure comes from the
Patterson-Sullivan density in a natural way, and ergodicity follows a
variation of the Hopf argument \cite{hopf}.
In the setting we study, the stable and unstable sets are not even locally
smooth and are not defined for a dense set of directions, but we are still
able to adapt this classical proof. 

\subsubsection*{Structure of the paper}

We first introduce Hilbert geometries and the central tools in Section
\ref{sec:preliminaries}.  In Section \ref{sec:asymptoticgeometry} we
gather lemmas on the 
asymptotic geometry 
and the compatible Busemann function, which apply to arbitrary Hilbert
geometries.  We then construct the Patterson-Sullivan density for the
universal cover of
a Benoist 3-manifold in Section
\ref{sec:pattersonsullivan} and prove the Shadow Lemma (Lemma
\ref{lem:shadowlem}), and in Section \ref{sec:psunique} we prove this
construction is canonical (Theorem \ref{thm:mainthm3}), with application to
growth rates of volumes of spheres and divergence of $\Gam$ (Theorem
\ref{thm:mainthm4}). Lastly, in Section \ref{sec:bowenmargulismeasure}, we
construct the Bowen-Margulis measure and complete the proof of the main
result, Theorem \ref{thm:mainthm1}. 

\subsubsection*{Acknowledgements} 
The author is very grateful to advisor Boris Hasselblatt, to postdoc
mentors Dick Canary and Ralf Spatzier, and to Aaron Brown and Micka\"el
Crampon for helpful conversations, emails, and feedback on the paper. The
author thanks the CIRM for support in the early stages of this work. Many
thanks to the referee for valuable feedback which immensely improved the
paper. The author was supported in part by NSF RTG grant 1045119. 

\section{Preliminaries}
\label{sec:preliminaries}

We say a domain $\Om\subset\R\P^n$ is \emph{properly convex} if $\Omega$
can be represented as a bounded convex set in some affine chart, and denote
by $\partial\Omega$ the topological boundary of $\Omega$ in $\mathbb
R\mathbb P^n$.  Define $H$ to be a \emph{supporting hyperplane} to a
properly convex $\Om\subset \RP^n$ if $H$ is a codimension 1 projective
subspace of $\RP^n$ which intersects $\bd\Om$ but not $\Om$. Then a
properly convex $\Om$ is \emph{strictly convex} if every supporting
hyperplane intersects $\bd\Om$ at a single point. 

For any properly convex domain $\Om$, fix an affine chart in which $\Omega$
is bounded and define the Hilbert $\Om$-distance between $x,y\in\Om$ as
follows: define $d_\Omega(x,x)=0$ and for $x\neq y$, let $\cl{xy}$ denote
the projective line in this affine chart uniquely determined by $x$ and $y$
and take $a,b\in \bd\Om$ to be the distinct intersection points of
$\cl{xy}$ with $\bd\Om$.  Then $d_\Om(x,y)=\frac12 |\log[a:x:y:b] |$ where
$[a:x:y:b]:=\frac{|a-y||x-b|}{|a-x||y-b|}$ is the Euclidean cross-ratio, a
projective invariant.  
It will
be useful to denote by $[xy]$ the segment of $\overline{xy}$ between $x$
and $y$ in $\overline{\Omega}:=\Omega\cup\partial\Omega$, and
$[xy)=[xy]\smallsetminus \{y\}$.  The cross-ratio of four projective lines
$L_1,L_2,L_3,L_4$ intersecting at one point is well-defined as
$[L_1:L_2:L_3:L_4]=[a_1:a_2:a_3:a_4]$ where $a_i\in L_i$ and
$a_1,a_2,a_3,a_4$ are collinear; this can be used to prove that $d_\Om$ is
a well-defined metric.  The group   $\Aut(\Om):=\{g\in \PSL(n+1,\R) \mid
g\Om=\Om\}$ is a subgroup of isometries of $(\Om, d_\Om)$.  The Hilbert
metric comes from a Finsler norm $F_\Omega$ defined on the tangent bundle
$T\Omega$ (see \cite{crampon_handbook}). The norm $F_\Omega$ is only
Riemannian when $\Omega$ is an ellipsoid and has the same regularity as the
boundary of the domain \cite{sociemethou_these, crampon_handbook}.
Projective lines are geodesic and are the only geodesics in the strictly
convex case, but in general geodesics are not always lines.  The metric
space $(\Omega,d_\Omega)$ is complete and the topology induced by the
metric $d_\Omega$ coincides with the ambient Euclidean topology on $\Omega$
in this affine chart. 

We say that $\Om$ in $\mathbb {RP}^n$ is \emph{divisible} if there exists a
discrete subgroup $\Gam$ of $\Aut(\Om)$ acting properly discontinuously and
cocompactly on $\Omega$. Also, $\Omega$ is {\em decomposable} if the cone
over $\Omega$ in $\mathbb R^{n+1}$ is decomposable, and {\em
indecomposable} if $\Omega$ is not decomposable (see \cite[Section
3]{marquis_handbook} for more details).  Let $M=\Om/\Gam$ denote the
quotient manifold. 

Since geodesics are not unique for the Benoist 3-manifolds, we define the
geodesic flow to be flowing along projective lines, as is the case when
$\Om$ is strictly convex.  More formally, let $\Om$ be a divisible properly
convex domain with dividing group $\Gam$ and quotient manifold $M$.  Let
$\ell_v\colon \R\to M$ be the projective line parameterized at unit Hilbert
speed, uniquely determined by $v\in T^1M$, the unit tangent bundle to $M$
for the Finsler norm $F_\Omega$.  The Finsler unit tangent bundle to
$\Omega$ is denoted $T^1\Omega$.  Then the Hilbert geodesic flow of
$\Omega$ is $\tilde\phi^t\colon T^1\Omega\to T^1\Omega$ defined by
$\phi^t(v)=\dot{\ell}_v(t)$, and this flow descends to the geodesic flow
$\phi^t$ on $T^1M$, the unit tangent bundle to the quotient.  Note that
this geodesic flow has the same regularity as the boundary of the universal
cover $\Omega$ (for more details, see \cite[Section
2.4]{crampon_handbook}).

Formally, a geodesic $\gamma$ for the Hilbert metric is a path in $\Omega$
or $M$ such that the length of any segment of $\gamma$ is equal to distance
between the endpoints.  On occasion we will parameterize $\gamma$ at unit
Hilbert speed, and treat $\gamma$ as a mapping from $\mathbb R$ to $\Omega$
or $M$ to take advantage of the parameterization.

\subsection{Busemann functions}

For any three points $x,y,z\in\Om$, we define the \emph{Busemann function} to be
\[
  \beta_z(x,y) = d_\Om(x,z)-d_\Om(y,z).
\]

Evidently, for all $z\in\Omega$, the function $\beta_z$ is anti-symmetric, meaning
$\beta_z(x,y)=-\beta_z(y,z)$, and satisfies the property of a cocycle, that is
$\beta_z(x,y)+\beta_z(y,w) = \beta_z(x,w)$, for all $x,y,w\in\Omega$.  Also $|\beta_z(x,y)|\leq
d_\Omega(x,y)$ by the triangle inequality. Lastly, since $\Gam$ is acting on $\Om$ by isometries,
$\beta_{\gam z}(\gam x,\gam y)=\beta_{z}(x,y)$ for all $\gam\in\Gam$.  Geometrically, $\beta_z(x,y)$
describes the signed distance between the Hilbert spheres centered at $z$ passing through $x$ and
$y$.

\begin{definition}
  \label{def:busemann}	      
  We wish to extend the Busemann functions to the boundary. 
  Define 
  \begin{align*} 
    &\beta^-_\xi(x,y) = \inf_{z_n\to\xi}
    \liminf_{n\to\infty}\beta_{z_n}(x,y),\\
    &\beta^+_\xi(x,y) = \sup_{z_n\to\xi} \limsup_{n\to\infty}\beta_{z_n}(x,y). 
  \end{align*}
\end{definition}

These functions exist and are bounded in absolute value by $d_\Om(x,y)$ for
all $x,y\in\Omega$ and $\xi\in\partial\Omega$.  It is straightforward to
verify for all $x,y,w\in \Omega$ and $\xi\in\partial\Omega$ that
$\beta_\xi^-(x,y)=-\beta_\xi^+(y,x)$, and 
\[
  \beta_\xi^-(x,y)+\beta_\xi^-(y,w)\leq \beta_\xi^-(x,w)\leq \beta_\xi^+(x,w)\leq
  \beta_\xi^+(x,y)+\beta_\xi^+(y,w).
\]
Since $\Gamma$ acts by isometries, 
we also have
$\beta_{\gamma\xi}^\pm(\gamma x,\gamma y)=\beta_\xi^\pm(x,y)$. 
Then if indeed $\beta^+_\xi=\beta^-_\xi$, we may define
$\beta_\xi=\beta_\xi^+$, and see that the anti-symmetric, cocycle, and
$\Gamma$-invariance properties of $\beta_z$ for $z\in\Omega$ extend to
$\beta_\xi$ for $\xi\in\partial \Omega$.  For such $\xi\in\partial\Omega$,
the \emph{horosphere} through $x\in\Om$ based at $\xi $ is the zero set of
$\beta_\xi(x,\cdot)$, denoted by $\mathcal H_\xi(x)$. 

\subsection{Boundary points}

Recall that a supporting hyperplane to $\Omega$ at a point $\xi$ in
$\partial\Omega$ is a projective hyperplane $H$ such that $H$ contains
$\xi$ and $H\cap \Omega=\varnothing$.  Borrowing language from convex
geometry, we introduce the following terms: 
\begin{definition}
  \label{def:smoothextremal}
  A point $\xi$ in $\partial \Omega$ is {\em smooth} if there is a unique
  supporting hyperplane to $\Omega$ at $\xi$.  The point $\xi$ in
  $\partial\Omega$ is {\em extremal} if $\xi$ is not contained in any open
  line segment inside $\partial \Omega$.  
\end{definition}

Note that smooth points in $\partial \Omega$ may not be $C^1$ points when
$\partial\Omega$ is treated as a curve in an affine chart. For the examples
of interest to this work, there will be a dense set of points in the
boundary for which the derivative is not defined.  By Benoist, the
complement of boundaries of properly embedded triangles in $\bd\Om$ is
exactly the set of {smooth extremal} points: 

\begin{proposition}[{\cite[Proposition 3.8]{Ben4}}]
  Let $\Gamma$ be a discrete, torsion-free subgroup of $\PSL(4,\mathbb R)$ which divides an
  indecomposable, divisible, properly convex domain in $\mathbb R\mathbb P^3$. Then 
  \begin{enumerate}[a)]
    \item For every nontrivial line segment $\sigma\subset\partial\Omega$, there exists a properly
      embedded triangle $\triangle$ such that $\sigma\subset \partial \triangle$.
    \item A point $x\in \partial \Omega$ is smooth if and only if $x$ is not the vertex of any
      properly embedded triangle in $\Omega$.
  \end{enumerate}
  \label{prop:benoist_on_bndy_pts}
\end{proposition}

It follows from Theorem \ref{thm:ben3mfldgeom} that smooth extremal points
are dense in $\partial \Omega$ in the setting of interest.  We will see
that these smooth extremal points carry the hyperbolic behavior of the
dynamics, and the Busemann functions will be well-defined for these points.

\begin{remark}
  We point out a particular feature that is special for the Benoist
  3-manifolds, and essential for our study. By Benzecri's thesis work,
  \cite{benzecri60}, there are no angular points in the boundary of the
  universal cover of a Benoist 3-manifold.  Benoist extracts the
  consequences of this result in Proposition \ref{prop:benoist_on_bndy_pts}
  and Theorem \ref{thm:ben3mfldgeom} part (b). If $\Omega/\Gamma$ is a
  Benoist 3-manifold, then every point in the boundary of $\Omega$ is
  either smooth or extremal;
  the only
  exceptions to smoothness are vertices of properly embedded triangles, and
  the only exceptions to extremality are points in the open edges of
  properly embedded triangles, and these cannot coincide (distinct
  properly embedded triangles have disjoint closures). 
  \label{rem:nonsmoothext}
\end{remark}

\subsection{Busemann densities}

We introduce here a nonstandard definition of 
Busemann 
densities to 
address issues with nonsmooth points in the boundary, where the
Busemann functions are not well-defined. 


\begin{definition} \label{def:busemanndensity}
  A {\em 
  Busemann density}
  of dimension $\alpha>0$ for $\Om$ is a family of finite Borel
  measures $\{\mu_x\}_{x\in\Omega}$ supported on $\partial\Omega$ which
  satisfy:
  \begin{itemize}
    \item (\emph{quasi-$\Gam$-invariance})  for all $\gam\in\Gam$,
      $\gam_\ast \mu_x = \mu_{\gam x}$,
      and
    \item (\emph{transformation rule}) 
      for all $x,y\in\Omega$ and $\xi\in \supp\mu_y,$ the measures
      $\mu_x$ and $\mu_y$ are absolutely continuous, and in particular their
      Radon--Nikodym derivative satisfies
      \[
	e^{-\alpha\beta_\xi^+(x,y)}\leq\frac{d\mu_{x}}{d\mu_y}(\xi)
	\leq e^{-\alp\beta^-_\xi(x,y)}.
      \]
  \end{itemize}		
\end{definition}

Note that if the Busemann function was well-defined for every point in
$\partial \Omega$, we would recover the standard definition of a
Busemann 
(conformal) density. 
To prove Theorem \ref{thm:mainthm3}, we will construct a 
Busemann density for which almost every point is smooth and extremal, and
consequently the Busemann functions will be defined almost everywhere 
and the density will be conformal in the usual sense.

\subsection{Shadow topology}
At times we will take advantage of the ambient Euclidean topology on
$\Omega$, represented as a bounded convex domain in an affine chart, and
the induced topology on $\partial \Omega$ and
$\overline{\Omega}=\Omega\cup\partial\Omega$.  We define another topology
on $\partial\Omega$ which interacts with the Hilbert geometry inside
$\Omega$ as follows. 

\begin{definition} \label{def:shadowtopology}
  Let $B_\Om(x,r)$ be the open metric $d_\Om$-ball about $x\in\Omega$ of
  radius $r$. Then the {\em shadow of radius $r$ from $x$ to $y$} is
  denoted by $\mathcal O_r(x,y)$, and is equal to the endpoints of
  projective rays based at $x$ which pass through the open metric ball 
  $B_\Omega(y,r)$.  These
  shadows generate a possibly basepoint-dependent topology on
  $\partial\Omega$ called the {\em shadow topology based at} $x$. More
  precisely, the shadow topology based at $x$ is the topology generated by
  the set 
  \[
    \{\mathcal O_r(x,y):y\in\Omega, r>0\}. 
  \]
\end{definition}

It is straightforward to confirm that this topology agrees with the
ambient Euclidean topology, and is therefore basepoint independent. 
If the properly convex domain $\Omega$ was
strictly convex with $C^1$ boundary, then a basis for this topology is the
set $\{\mathcal O_r(x,y):y\in\Omega\}$ for any fixed $r>0$. We cannot
conclude as much in the nonstrictly convex setting, but we will see that
this topology is still well-behaved near the smooth extremal points, which
will suffice for the development of the Patterson-Sullivan theory. 

%

\subsection{Regular vectors}

For any vector $v\in T^1\Omega$, there is a unique oriented projective line $\ell_v$ determined by
$v$, and we let $v^-$ and $v^+$ denote the intersections of $\ell_v$ in $\partial \Omega$ in
backward and forward time, respectively. 

\begin{definition}
  A vector $v\in T^1\Omega$ is {\em regular} if both $v^-$ and $v^+$ are smooth extremal points.
  The set of {regular vectors} in $T^1\Om$ is denoted by $T^1\Omega_{\reg}$. Regularity is preserved
  by projective transformations, so a vector in $T^1M$ is regular if any lift in $T^1\Omega$ is
  regular, and $T^1M_{\reg}$ is the set of all regular vectors in $T^1M$. 
  \label{def:regularvectors}
\end{definition}

\subsection{Standing assumptions} 

In Section \ref{sec:asymptoticgeometry},
we take $\Omega$ to be an arbitrary properly convex set in real projective
space in unspecified dimension. In the remaining Sections
\ref{sec:pattersonsullivan}, \ref{sec:psunique}, and
\ref{sec:bowenmargulismeasure}, we assume $\Omega$ is a 
nonstrictly convex, properly convex, divisible,
indecomposable domain in real projective space of dimension 3, 
with discrete torsion-free
dividing group $\Gamma$,  so that $M=\Om/\Gam$ is a Benoist 3-manifold.  

Throughout the paper, we fix an affine chart in which $\Omega$
is bounded, and work with $\Omega$ in this affine chart. 

\section{Asymptotic geometry}
\label{sec:asymptoticgeometry}

%
In this section, we will prove some straightforward lemmas on the shadow
topology and Busemann functions for a Hilbert geometry in any dimension.
These results are likely well-understood by experts; the proofs are included for
completeness. 

\subsection{The shadow topology}

The following lemma confirms that for an extremal point $\xi$ in
$\partial \Omega$, shadows of a fixed radius generate the local topology at
$\xi$. 
This fact requires that $\xi$ is both smooth and extremal. 

\begin{lemma}
  Let $y_n\in\Omega$ be a sequence of points converging along a projective
  line to $\xi\in\partial \Omega$. Then $\xi$ is an extremal point in
  $\partial\Omega$ if and only if for all $r<0,x\in\Omega$, 
  \[
    \bigcap_{n\in\mathbb N} \mathcal O_r(x,y_n) =\{\xi\}.
  \]
  \label{lem:shadows_converge_to_smoothext}
\end{lemma}

\begin{proof}
  Let $\eta$ be a point in $\partial\Omega$ distinct from $\xi$, and let
  $y=y_0$. For each $n$, let $x_n$ be a closest point to $y_n$ on the
  projective line from $x$ to $\eta$ in the Hilbert metric. Let $a_n$,
  $b_n$ be such that $d_\Omega(x_n,y_n)=\frac12\log [a_n:x_n:y_n:b_n]$. 
  Since $\xi$ is extremal, $y_n$ converging to $\xi$ implies the same for
  $b_n$, hence the Hilbert distance between $x_n$ and $y_n$ goes to
  infinity as $n$ grows. Hence for large $n$, the projective ray $(x\eta)$
  does not intersect the ball $B_\Omega(x,y_n)$, and $\eta$ is not in
  $\mathcal O_r(x,y_n)$. 
  
  Conversely, see that if $\xi$ is not extremal, then there is an open line
  segment contained in the shadow $\mathcal O_r(x,y_n)$ for all $n$. 
\end{proof}

\subsection{The Busemann function and horospheres}
In this subsection, we verify some regularity properties of the Busemann
function.  

\begin{lemma}
  The Busemann function $\beta_\xi$ is well-defined on smooth points $\xi$
  in $\partial\Omega$, and $\beta_\xi(x,y)$ varies continuously over the
  inputs $x$ and $y$ in $\Omega$. 

  More specifically, we have the following geometric description of the
  Busemann function: for any $x,y\in\Omega$, and any smooth boundary point
  $\xi\in\partial\Omega$, let $H_\xi$ be the supporting hyperplane to
  $\Omega$ at
  $\xi$, and let $x^-,y^-$ be the intersection points of the lines 
  $\overline{x\xi},\overline{y\xi}$ respectively with $\partial \Omega$
  which are not $\xi$. If $x^-\neq y^-$, 
  let $q(x,y,\xi)$ be the unique intersection point of the line
  $\overline{x^-y^-}$ 
  with the hyperplane $H_\xi$ in
  projective space. Then if $x,y$ and $\xi$ are not collinear, 
  \[
    \beta_\xi^-(x,y)=\beta_\xi^+(x,y)=\frac12\log
    [\overline{x^-q}:\overline{xq}:\overline{yq}:\overline{\xi q}]
  \]
  and otherwise, 
  \[
    \beta_\xi^-(x,y)=\beta^+(x,y)=\frac12\log[x^-:x:y:\xi].
  \]
  \label{lem:busemann}
\end{lemma}

\begin{figure}[t]
  \centering
  \begin{tikzpicture}[scale=1] 
    \def\sep{1.2}

    \draw (0,-.2) coordinate (a) (2,-.5) coordinate (b) (1,1.25) coordinate
    (c); 
    \draw (a) .. controls ++ (.2,-.6) and ++(-.4,-.6) .. (b)
    coordinate[pos=.65] (xi);
    \draw (b) .. controls ++ (.5,.6) and ++(1.2,-.3) .. (c)
    coordinate[midway] (y-);
    \draw (a) .. controls ++ (-.1,.4) and ++ (-1.2,.1) .. (c)
    coordinate[pos=.95] (x-);

    \draw (3.5,1.4) node {$\Omega\cap P$};
    \draw (xi) node[circle, inner sep=\sep, fill=black,
    label={below:{$\xi$}}] {};
    \draw (y-) node[circle, inner sep=\sep, fill=black, label={above
    right:{$y^-$}}] {};
    \draw (x-) node[circle, inner sep=\sep, fill=black,
    label={above:{$x-$}}] {};

    \draw[name path=findq1] (xi)++(4,0) node[below] {$ L_\xi$} 
    -- (xi) ;
    \draw[name path=findq2, draw=none] [add=0 and 2.85] (x-) to (y-);
    \path [name intersections={of=findq1 and findq2,by=q}];
    \draw (q) node[circle, inner sep=\sep, fill=black, label={below:{$q$}}]
    {};
    \draw (x-)--(q);

    \draw[name path=XI] (x-)-- 
    coordinate[pos=.3] (xbar) 
    coordinate[pos=.6] (x) 
    (xi);
    \draw[name path=YXI] (xi) -- 
    (y-);
    \draw (x) node[circle, inner sep=\sep, fill=black, label={left:{$x$}}]
    {};
    \draw (xbar) node[circle, inner sep=\sep, fill=black,
    label={left:{$\bar{x}$}}] {};


    \draw[name path=CR1] (x) --(q);
    \draw[name path=CR2] (xbar)--(q);
    \path [name intersections={of=YXI and CR2,by=y}];
    \node [circle, fill, inner sep=\sep,label=94:$y$] at (y) {};

  \end{tikzpicture}
  \begin{tikzpicture}[scale=1] 
    \def\sep{1.2}

    \draw (3.5,1.4) node {$\Omega\cap P_n$};

    \draw (0,-.2) coordinate (a) 
    (2,-.5) coordinate (b)
    (1,1.25) coordinate (c); 
    \draw (a) .. controls ++ (.2,-.6) and ++(-.4,-.6) .. (b) 
    coordinate[pos=.8] (xn+)
    coordinate[pos=.4] (yn+);
    \draw (b) .. controls ++ (.5,.6) and ++(1.2,-.3) .. (c)
    coordinate[midway] (yn-);
    \draw (a) .. controls ++ (-.1,.4) and ++ (-1.2,.1) .. (c)
    coordinate[pos=.95] (xn-);

    \draw (xn+) node[circle, inner sep=\sep, fill=black,
    label={below right:{$x_n^+$}}] {};
    \draw (yn-) node[circle, inner sep=\sep, fill=black, label={above
    right:{$y_n^-$}}] {};
    \draw (yn+) node[circle, inner sep=\sep, fill=black, label={below
    left:{$y_n^+$}}] {};
    \draw (xn-) node[circle, inner sep=\sep, fill=black,
    label={above:{$x_n^-$}}] {};

    \draw[thick, name path=findq1, add= 0 and 3.3] (yn+) to (xn+) ;
    \draw[name path=findq2, draw=none] [add=0 and 2.85] (xn-) to (yn-);
    \path [name intersections={of=findq1 and findq2,by=qn}];
    \draw (qn) node[circle, inner sep=\sep, fill=black, label={below:{$q_n$}}] {};
    \draw[] (xn-)--(qn);

    \draw[name path=XZN] (xn-)-- 
    coordinate[pos=.35] (xn) 
    coordinate[pos=.5] (x)
    (xn+);
    \draw[name path=YZN] (yn+) --  (yn-);
    \path [name intersections={of=XZN and YZN,by=zn}];
    \draw (xn) node[circle, inner sep=\sep, fill=black,
    label={left:{$x_n$}}] {};
    \draw (zn) node[circle, inner sep=\sep, fill=black,
    label={left:{$z_n$}}] {};
    \draw (x) node[circle, inner sep=\sep, fill=black,
    label={left:{$x$}}] {};


    \draw[name path=CR1] (xn) --(qn);
    \path [name intersections={of=YZN and CR1,by=y}];
    \node [circle, fill, inner sep=\sep,label=93:$y$] at (y) {};
    \draw (zn) -- (qn);

  \end{tikzpicture}

  \caption{For the proof of Lemma \ref{lem:busemann}.
  In the left panel, we take the 2-dimensional intersection of $\Om$ with
  the projective plane $P$ determined by
  $x,\xi,$ and $y$.  In the right panel, we take a sequence 2-dimensional
  intersections of 
  $\Om$ with the projective plane $P_n$ 
  determined by the projective lines $\overline{x z_n}$ and $\cl{y z_n}$,
  and see that
  $\beta_{z_n}(x,y)=\beta_{z_n}(x,x_n)=\frac12\log[x_n^-:x:x_n:x_n^+]$ 
  where $x_n$ is as pictured.  In
  Lemma \ref{lem:busemann} we confirm that if $\xi$ is smooth, then the
  image on the right converges to the image on the left, and 
  $\beta_\xi(x,y)=\frac12\log[x^-:x:\bar{x}:\xi]$
  as pictured in the left panel.}
  \label{fig:busemann}
\end{figure}

\begin{proof}
%
  Suppose first that $x,y,$ and $\xi$ are not collinear, so there is a unique
  projective plane $P$ containing $x,y,$ and $\xi$. 
  Since $\xi$ is smooth in $\partial\Omega$, then $\xi$ is also
  smooth in $\partial\Omega\cap P$, and there is a unique supporting
  hyperline $L_\xi$ to $\partial\Omega\cap P$ at $\xi$ 
  For each $n$ sufficiently large,
  let $P_n$ be the projective plane containing the three points $x,y$ and
  $z_n$. As $n$ goes to infinity, $P_n$ converges to $P$ in the
  Gromov-Hausdorff sense. 

  In the projective
  plane $P_n$, let $y_n^+$ and
  $y_n^-$ be the intersection points of the projective line
  $\overline{yz_n}$ with the boundary $\partial\Omega\cap P_n$, such that
  $y_n^-$ is closer to $y$ than $z_n$ in the affine metric, and choose
  $x_n^-,x_n^+$ similarly with respect to $x$. Then in the projective
  subspace $P_n$, the lines $\overline{x_n^-y_n^-}$ intersect at some point
  $q_n$.  A picture detailing this set-up is available in Figure
  \ref{fig:busemann}.

  A quick calculation confirms the cross-ratio has the property that
  $[a:x_1:x_2:b][a:x_2:x_3:b]=[a:x_1:x_3:b]$. 
  Then since the cross-ratio of four lines is well-defined and
  $\overline{q_nx_n^\pm}=\overline{q_ny_n^\pm}=\overline{x_n^\pm y_n^\pm}$, 
  \begin{multline*}
    \beta_{z_n}(x,y) = 
    \frac12\log
    [\overline{x_n^-y_n^-}  :  \overline{x q_n} : \overline{z_nq_n} :
    \overline{x_n^+y_n^+}]
    -
    \frac12\log
    [\overline{x_n^-y_n^-} : \overline{y q_n} : \overline{z_n q_n} :
    \overline{x_n^+y_n^+}]  \\
    = \frac12\log 
    [\overline{x_n^-y_n^-}  :  \overline{x q_n} : \overline{y q_n} :
    \overline{x_n^+y_n^+}]
  \end{multline*}
  
%
  Now, as $z_n$ converges to $\xi$, the points $x_n^-$ and $y_n^-$ converge
  to $x^-$ and $y^-$ respectively,
  hence the lines
  $\overline{x_n^-y_n^-}$ converge to the line $\overline{x^-y^-}$. Both
  $y_n^+$ and $x_n^+$ converge to the smooth point $\xi$ and are contained
  in the plane $P_n$ which converges to the plane $P$, 
  so the lines
  $\overline{y_n^+x_n^+}$ must converge to the unique supporting projective
  line $L_\xi$ to $\xi$ in $\partial\Omega\cap P$. It follows that the
  sequence of points $q_n$ in $P_n$, which is the unique intersection
  point of the lines $\overline{x_n^-y_n^-}$ and $\overline{x_n^+y_n^+}$,
  converges to the intersection point $q$ of the lines $\overline{x^-y^-}$
  and $L_\xi=\overline{\xi q}$. 
  The conclusion follows, in this case where
  $x,y$ and $\xi$ are not collinear. 

  If $x,y,$ and $\xi$ are collinear, then let $w$ be any other point in
  $\Omega$ which is not collinear. Then a short calculation confirms that 
  $\beta_{z_n}(x,y)=\beta_{z_n}(x,w)+\beta_{z_n}(y,w)$ converges to
  $\beta_\xi(x,w)+\beta(w,y)=\frac12\log[x^-:x:y:\xi]$ as desired. Finally,
  it is now clear to see geometrically that the Busemann functions at fixed
  $\xi$ vary continuously in the inputs $x,y\in\Omega$. 
%
\end{proof}

Thus, we have:

\begin{corollary}
  When $\xi$ is smooth, the Busemann functions $\beta_\xi$ satisfy the
  anti-symmetric, cocycle, and $\Gamma$-invariance properties, as discussed in
  Definition \ref{def:busemann}. 
  \label{cor:smoothbusemann}
\end{corollary}

\begin{lemma}
  The Busemann functions restricted to smooth points in $\partial\Omega$
  are continuous.
  \label{lem:continuitybusemann}
\end{lemma}

\begin{proof}
  Since $\xi$ is a smooth point in $\bdv\Om$, the Busemann functions at
  $\xi$ are well-defined by Lemma \ref{lem:busemann}.  
  Let $x^-$, $y^-$ be the other intersection points of $\cl{\xi x}$,
  $\cl{\xi y}$ with $\bd\Om$, respectively. Let $H_{\xi}$ denote the unique
  supporting hyperplane to $\Om$ at $\xi$. We proceed under the assumption
  that 
  $x^-\neq y^-$, so that $x,y,$ and $\xi$ determine a projective plane $P$,
  though the arguments easily generalize to the case where $x,y,$ and
  $\xi$ are collinear by Lemma \ref{lem:busemann}. 
  To complete the setup, let $q$ again be the unique intersection of the
  line $\cl{x^-y^-}$ with the hyperplane $H_\xi$ in $P$. 
  See Figure \ref{fig:continuitybusemann} for clarity. 

  Let $\xi_n$ in $\partial\Omega$ be a sequence of smooth boundary points
  converging to $\xi$ and let $H_{\xi_n}$ be the unique supporting
  hyperplanes to $\Omega$ at $\xi_n$. Take $x_n^-$ to be the other
  intersection point of $\cl{x\xi_n}$ with $\bd\Om$ and $y_n^-$ the other
  intersection point of
  $\cl{y\xi_n}$ with $\bd\Om$. Let $P_n$ be any projective plane containing
  $x,y$ and $\xi_n$; since $\xi_n$ converges to $\xi$, then $P_n$ also
  converges to the plane $P$ containing $x,y$ and $\xi$. Lastly, take
  $q_n$ to be the intersection point in $P_n$ of the line
  $\overline{x_n^-y_n^-}$ with the supporting hyperplane $H_{\xi_n}$.

  By Lemma \ref{lem:busemann}, to show that $\beta_{\xi_n}(x,y)$ converges
  to $\beta_\xi(x,y)$, it suffices to show that $q_n$ converges to
  $q$. The points $q_n$ in the compact complement of
  $\Omega$ must accumulate on some point $q'$, which must be in the plane
  $P$ because the planes $P_n$ converge to $P$, and must also be on the
  line $\overline{x^- y^-}$ so it cannot equal $\xi$. The lines 
  $\overline{\xi_n
  q_n}$ are disjoint from $\Omega$, hence the same holds for the limiting
  line $\overline{\xi q'}$. Then $q'$ lies on the unique supporting
  hyperline to $\Omega\cap P$ at $\xi$, and must equal the unique
  intersection of $\overline{x^-y^-}$ with this line. 
\end{proof}


%

\begin{figure}[t]
  \centering
  \begin{tikzpicture}[scale=1.5] 
    \def\sep{1.2}

    \draw (0,-.2) coordinate (a) 
    (2,-.5) coordinate (b)
    (1,1.25) coordinate (c); 
    \draw (a) .. controls ++ (.2,-.6) and ++(-.4,-.6) .. (b) coordinate[pos=.65]
    (xi);
    \draw (b) .. controls ++ (.5,.6) and ++(1.2,-.3) .. (c) coordinate[pos=.62]
    (y-) coordinate[pos=.87] (yn-);
    \draw (a) .. controls ++ (-.1,.4) and ++ (-1.2,.1) .. (c) coordinate[pos=.95]
    (x-) coordinate[pos=.7] (xn-);

    \draw (xi) node[circle, inner sep=\sep, fill=black, label={below left:{$\xi$}}] {};
    \draw (y-) node[circle, inner sep=\sep, fill=black, label={above right:{$y^-$}}] {};
    \draw (x-) node[circle, inner sep=\sep, fill=black, label={above:{$x^-$}}] {};

    \draw[name path=findq1, draw = none] (xi)++(5.5,0)  -- (xi) ;
    \draw[name path=findq2, draw=none] [add=0 and 2.85] (x-) to (y-);
    \path [name intersections={of=findq1 and findq2,by=q}];

    \draw (xi) -- (q) node[midway, below] {$ H_\xi$}; 
    \draw (q) node[circle, inner sep=\sep, fill=black, label={below:{$q$}}] {};
    \draw (x-)--(q);

    \draw[name path=XI] (x-)-- coordinate[pos=.3] (o) coordinate[pos=.6] (x) 
    (xi);
    \draw[name path=ETA] (xi) -- coordinate[very near start] (eta)  (y-);
    \draw (x) node[circle, inner sep=\sep, fill=black, label={left:{$x$}}] {};

    \draw[name path=CR2] (x)--(q);
    \path [name intersections={of=ETA and CR2,by=y}];
    \node [circle, fill, inner sep=\sep,label=below right:$y$] at (y) {};

    \begin{scope}[dashed]

      \draw (yn-) node[circle, inner sep=\sep, fill=black, label={above
      right:{$y_n^-$}}] {};
      \draw (xn-) node[circle, inner sep=\sep, fill=black, label={above
      left:{$x_n^-$}}] {};
      \draw[add = 0 and .7] (xn-) to (x) coordinate (xin); 
      \draw (xin) node[circle, inner sep=\sep, fill=black, label={below
      right:{$\xi_n$}}] {};

      \draw[name path=findp1, draw=none] (xin)++(25:5.5) -- (xin);
      \draw[add=0 and 5.5, name path=findp2, draw=none] (xn-) to (yn-);

      \path [name intersections={of=findp1 and findp2,by=p}];
      \draw (xin)--(p) node[circle, fill, inner sep=\sep,
      label=right:$q_n$] {} --(xn-);

      \draw (xin)--(yn-);

      \draw[add=0 and .07] (p) to (y) {};

    \end{scope}

  \end{tikzpicture}
  \caption{
    For the proof of Lemma \ref{lem:continuitybusemann}. 
    For clarity and simplicity, the figure only depicts the case where 
    $\Omega$ is two-dimensional and $x,y$ are such 
    that $\beta_\xi(x,y)=0$. By Lemma \ref{lem:busemann}, to
    show that the Busemann functions $\beta_{\xi_n}(x,y)$ converge to
    $\beta_\xi(x,y)$ as $\xi_n$ converges to $\xi$, it suffices to show
    that $q_n$ converges to $q$. 
}
  \label{fig:continuitybusemann}
\end{figure}

Since horospheres are zero sets of the Busemann function, we have:

\begin{corollary}
  Horospheres based at smooth boundary points are globally defined,
  continuous, and vary continuously over smooth points. 
  \label{cor:horospheres}
\end{corollary}

\section{Patterson-Sullivan Theory}
\label{sec:pattersonsullivan}
%
In this section, we construct the Patterson-Sullivan density for the
universal cover of a Benoist 3-manifold.
The density is named for the independent work of Patterson and Sullivan in
negative curvature and has since been generalized to many settings,
including rank one manifolds \cite{patterson,sullivan79,Kn97}.  Theorems
\ref{thm:mainthm3} and \ref{thm:mainthm4} follow the study of these
measures and their properties.  To generalize the results beyond dimension
three, we need deeper understanding of the geometry of the flats and
hyperbolicity of the group in higher dimensions. 

\subsection{Poincar\'e Series and the critical exponent}
\label{sec:criticalexponent}

The \emph{critical exponent}, $\delta_\Gam$, of a group $\Gam$ acting discretely, properly
discontinuously, and by isometries on $(\Om,d_\Om)$ is the critical value of $0\leq s\in\R$ for the
\emph{Poincar\'e series},
\[
  P(x,y,s)=\sum_{\gam\in\Gam}e^{-sd_\Om(x,\gam y)}. 
\]

The group $\Gam$ is of \emph{divergent type} if $P(x,y,\delta_\Gam)$ diverges and \emph{convergent
type} if $P(x,y,\delta_\Gam)$ converges.  It is straightforward to verify that convergence of
$P(x,y,s)$ does not depend on $x$ or $y$ by
the triangle inequality and that we can realize $\delta_\Gam =
\displaystyle\limsup_{t\to\infty}\frac1t\log N_\Gam(x,t)$ where $N_\Gam(x,t):=\#\{\gam\in\Gam\mid d_\Om(x,\gam
x)\leq t\}$ for some $x\in\Omega$.  
By previous work we have that $\delta_\Gam>0$ \cite{braytop}. 
When $\Gamma$ is a discrete group acting on $\Omega$ with finite co-volume, 
$\delta_\Gam \leq \dim(\Om)-1$, 
with equality if and only if $\Om$ is the ellipsoid; this generalizes a
result of Crampon for the strictly convex case \cite{BMZ,Cr09}. In our
setting where the quotient $\Omega/\Gamma$ is compact, the inequality
$\delta_\Gamma\leq \dim(\Omega)-1$, without the rigidity statement, follows
quickly from a theorem of Tholozan that the volume growth entropy is
bounded above by $\dim(\Omega)-1$ \cite{tholozan17}.  Although the theorem
of Tholozan requires no group action at all, in the cocompact case, the
critical exponent and volume growth entropy coincide, hence the result can
be applied to produce a bound on the critical exponent. 

\subsection{Patterson-Sullivan densities}
\label{sec:confdensities}

We will now prove that a 
Busemann density exists for the
universal cover of a Benoist 3-manifold. The argument will depend on
features of the Benoist 3-manifolds discussed in Remark
\ref{rem:nonsmoothext}.

\begin{proposition}
  There exists a 
  Busemann density of dimension $\delta_\Gam>0$ on $\bdv\Om$, called a
  Patterson-Sullivan density. 
  \label{prop:existenceconfdensity}
\end{proposition}

\begin{proof}
  The construction follows Patterson and Sullivan \cite{patterson,
  sullivan79}. For $s>\delta_\Gam$, choose an observation point $o\in\Om$
  for the measures and for the visual boundary. For each $x\in\Om$ define a
  measure on $\cl{\Om}$ by 
  \[
    \mu_{x,s} = \frac1{P(o,o,s)}\sum_{\gam\in\Gam} e^{-sd_\Om(x,\gam o)}
    \delta_{\gam o}
  \]
  where $\delta_p$ is the Dirac mass at $p$. Note that for $s>\delta_\Gam$,
  $\mu_{x,s}$ is supported on $\Om$. Also, by definition of the critical
  exponent, if $s>\delta_\Gam$ then $P(x,y,s)$ is finite for all
  $x,y\in\Om$ so $\mu_{x,s}(\cl{\Om})= P(x,o,s)/P(o,o,s)<\infty$. By
  compactness of $\cl{\Om}$ we may extract a weak limit by choosing a
  convergent subsequence as $s$ decreases to $\delta_\Gam$ to obtain a
  finite nontrivial measure,
  \[
    \mu_x = \lim_{s_n\to\delta_\Gam^+} \mu_{x,s_n}.
  \]
  If the Poincar\'e series diverges at $\delta_\Gamma$ ($\Gam$ is of
  divergent type), then the total mass of $\mu_{x,s}$ is pushed to
  $\bdv\Om$ as $s$ decreases to $\delta_\Gam$ and $P(o,o,s)\to\infty$.  At
  the limit, $\supp\mu_x\subset \bdv\Om$.  If the Poincar\'e series
  converges at $\delta_\Gam$ ($\Gam$ is of convergent type), then we follow
  Patterson's method for Fuchsian groups which generalizes to any manifold
  group \cite{patterson}. First, he showed it is possible to constuct an
  increasing function $f\colon\R^+\to\R^+$ with subexponential growth: that
  is, for all $\ep<0$, there exists an $x_0(\ep)>0$ such that for all
  $x>x_0$, $y>0$ 
  \[
    f(x+y)\leq f(x) e^{y\cdot \ep},
  \]
  and the modified Poincar\'e series 
  \[
    P_f(x,y,s) = \sum_{\gam\in\Gam} f(d_\Om(x,\gam y))e^{-sd_\Om(x,\gam
    y)}
  \]
  has the same critical exponent $\delta_\Gamma$ and diverges at
  $s=\delta_\Gam$ \cite[Lemma 3.1]{patterson}. Then we denote by $\mu_x^f$
  a weak limit as $s$ decreases to $\delta_\Gamma$ of 
  \[
    \mu_{x,s}^f = \frac1{P_f(o,o,s)}\sum_{\gam\in\Gam} f(d_\Om(x,\gam
    o))e^{-sd_\Om(x,\gam o)}\delta_{\gam o}.
  \]
  Taking $f\equiv1$ recovers $\mu_x$, so we will check that these measures
  satisfy the definition of a Busemann density for the case that $\Gam$ is
  convergent.

  We remark first that $P_f(o,o,s)$ exhibits the same convergence and
  divergence behavior as $P(o,o,s)$ for $s\neq\delta_\Gam$ so $\mu^f_{x,s}$
  will be a finite nontrivial measure supported on point masses in $\Om$
  much like $\mu_{x,s}$. Taking a weak-limit then produces a finite
  nontrivial measure $\mu_x^f$ supported on $\bdv\Om$ by the divergence of
  $P_f(o,o,s)$ as $s$ decreases to $\delta_\Gamma$.  Moreover, for any
  Borel measurable set $A\subset \cl{\Om}$, 
  \begin{multline*}
    \mu_{x,s}^f(\gam^{-1}A)  = \frac1{P_f(o,o,s)}\sum_{g\in\Gam}
    f(d_\Om(x,go))e^{-sd_\Om(x,go)}\delta_{go}(\gam^{-1}A) \\ =
    \frac1{P_f(o,o,s)}\sum_{\gam g\in\Gam} f(d_\Om(\gam x,\gam g
    o))e^{-sd_\Om(\gam x,\gam go)} \delta_{\gam go}(A) = \mu_{\gam
    x,s}^f(A).
  \end{multline*}
  Then the quasi-$\Gam$-invariance property from Definition
  \ref{def:busemanndensity} holds for any weak limit $\mu_{x}^f$.  Since
  $\mu^f_{x,s}$ is supported on countably many point masses in $\Om$ for
  $s>\delta_\Gam$, we compute 
  \begin{align*}
    \frac{d\mu_{x,s}^f}{d\mu_{y,s}^f}(\gam o) = \frac{\mu_{x,s}^f(\gam
  o)}{\mu_{y,s}^f(\gam o)} = \frac{f(d_\Om(x,\gam o))e^{-sd_\Om(x,\gam
  o)}}{f(d_\Om(y,\gam o))e^{-s d_\Om(y,\gam o)}} = \frac{f(d_\Om(x,\gam
  o))}{f(d_\Om(y,\gam o))} e^{-s \beta_{\gam o}(x,y)}.
\end{align*}
As $s\to\delta_\Gam$, indeed $\supp\mu^f_{x,s},\supp\mu^f_{y,s}$ is
pushed to $\bdv\Om$.  By the increasing and subexponential properties of
$f$, for all $\ep>0$ we have that for all $\gam o$ such that $d_\Om(\gam
o,y)$ is sufficiently large,
\[
  f(d_\Om(x,\gam o))\leq f(d_\Om(y,\gam o)+d_\Om(x,y))\leq f(d_\Om(y,\gam
  o) ) \;e^{d_\Om(x,y)\cdot \ep}.
\] 
Then for any $\gamma\in\Gamma$ such that $d(x,\gamma o)$ is sufficiently
large, 
\[
  e^{-d_\Omega(x,y)\epsilon} e^{-s\beta_{\gamma o}(x,y)}\leq 
  \frac{d\mu_{x,s}^f}{d\mu_{y,s}^f}(\gamma o) 
  \leq e^{d_\Omega(x,y)\epsilon} e^{-s \beta_{\gamma o}(x,y)}.
\]
To extend the Radon-Nikodym derivative to the limit, let $D$ be any compact
fundamental domain containing the fixed point $o$, and let
$\xi\in\partial\Omega$ be arbitrary. If $\xi$ is smooth, by minimality of
$\Gamma$ acting on $\partial\Omega$ \cite[Proposition 3.10]{Ben4} there
exists a sequence of group elements $\gamma_n$ such that $\gamma_no$
converges to $\xi$. Then apply Lemma \ref{lem:busemann} to the smooth point
$\xi$ to conclude
$\beta_{\gamma_n o}(x,y)$ converges to the well-defined Busemann function
$\beta_\xi(x,y)$ as desired. 

If $\xi$ is not smooth, then $\xi$ must be extremal by Benoist's structure
theorems, as discussed in Remark \ref{rem:nonsmoothext}.
Cover the projective ray from $o$ to
$\xi$ by orbits $\gamma_n D$ of $D$ under the group $\Gamma$, so
$d_\Omega(\gamma_no,o)$ diverges. For each
$n$, choose a point $z_n$ on the projective ray from $x$ to $\xi$ that lies
in $\gamma_n D$. Then $z_n$ converges to $\xi$ along a projective ray, and
by the triangle inequality,
\[
  |\beta_{\gamma_no}(x,y)-\beta_{z_n}(x,y)|\leq 2d_\Omega(\gamma_no,z_n).
\]
%
Since $\xi$ is extremal, we can in fact choose a sequence of points $z_n$
on the projective ray from $o$ to $\xi$ such that
$d_\Omega(\gamma_no,z_n)$ converges to zero, and hence any accumulation of
$\beta_{\gamma_n o}(x,y)$ as $n$ goes to $\infty$ is bounded above and
below by $\beta_\xi^\pm(x,y)$, as desired.
To construct the sequence, let $\xi_n$ be the
endpoint in $\partial \Omega$ of the projective ray from $o$ passing through
$\gamma_no$. It suffices to show that $\xi_n$ converges to $\xi$. By
contrapositive, suppose $\xi_n$ does not converge to $\xi$.
Choose $R$ larger than twice the diameter of the compact fundamental domain $D$.
Since $\xi$ is extremal, by Lemma \ref{lem:shadows_converge_to_smoothext},
the shadows $\mathcal O_R(o,y)$ around points $y$ on the projective ray
$(o\xi)$ from $o$ to $\xi$ form a neighborhood basis for $\xi$. The
assumption that $\xi_n$ does not converge to $\xi$ implies that there
exists a $T$ and a subsequence $\xi_{n_j}$ such that for all $y$ on the
projective ray $(o\xi)$ distance at least $T$ from $o$, then $\xi_{n_j}$ is
not in $\mathcal O_R(o,y)$, and equivalently, $\gamma_{n_j}o$ is not in the
ball $B_\Omega(y,R)$. But then, for $j$ large, $\gamma_{n_j}o$ cannot be in
the image of the fundamental domain $\gamma_{n_j}D$, and we conclude the
argument. 
%
\end{proof}

\begin{remark}
  \label{rem:psmeasures}
  The Patterson-Sullivan measures are Borel measures on
  $\partial \Omega$ and have full support by quasi-$\Gamma$-invariance and
  minimality of the action of $\Gamma$ on $\partial \Omega$
  \cite[Proposition 3.10]{Ben4}. 
\end{remark}

\subsection{The Shadow Lemma and applications}

In this subsection we prove Sullivan's Shadow Lemma in the setting of interest \cite{sullivan79}.

\subsubsection{Geometric lemmas}

Define $\gam \in \Aut(\Om)$ to be {\em hyperbolic} if $\gam$ has an attracting fixed point and a
repelling fixed point in $\bd\Om$, denoted $\gam^+$ and $\gam^-$, which are both smooth and
extremal, and $\gamma$ has no other fixed points in $\overline{\Omega}$.  This definition diverges
from the classical definition that the translation length of $\gam$ is positive and realized in
$\Om$, which is a consequence but not equivalent.  We choose this definition in this setting to
separate stabilizers of triangles from group elements that act hyperbolically with north-south
dynamics, since both such isometries have positive translation length realized in $\Om$. We will
need a proposition from the topological study of the Benoist 3-manifolds, which is straightforward
given Theorem \ref{thm:ben3mfldgeom} of \cite{Ben4}:

\begin{proposition}[{\cite{braytop}}]
  If $M=\Om/\Gam$ is a Benoist 3-manifold then $\Gam$ is the disjoint union of hyperbolic isometries
  and stabilizers of properly embedded triangles. There are infinitely many conjugacy classes of
  hyperbolic group elemments.
  \label{prop:toppaperhyperbolic}
\end{proposition}

The immediate goal is to prove the following geometric proposition, similar to that in
\cite{ballman95}, as needed for the Shadow Lemma.

\begin{proposition}
  Fix $x\in\Omega$.  For any two noncommuting hyperbolic isometries $g,h$ preserving $\Om$ and $O$ a
  sufficiently small neighborhood of $h^+$, there exists an $R$  large and $M\in \N$ such that for
  all $r\geq R$ and all $y\in\Om$, either $h^M O\subset \mathcal O_r(y,x) $ or $g^Mh^M O\subset
  \mathcal O_r(y,x)$.  
  \label{prop:geometricprop}
\end{proposition}

We first prove two geometric lemmas.  Let $\mathcal C A$ be the convex hull of a subset $A$ in our
affine chart for $\Om$. 

\begin{lemma}
  If $h$ is a hyperbolic isometry then for any open sets $O^+\subset \bd\Om$ containing $h^+$ and
  $\mathcal O^-$ containing $h^-$, there exists an $N\in\N$ such that for all $n\geq N$,
  \[
    h^n (\cl{\Om} \smallsetminus \mathcal C O^-) \subset O^+ \text{ and }
    h^{-n}(\cl{\Om}\smallsetminus \mathcal C O^+)\subset O^-.
  \]
  \label{fct:hyperbolicisoms}
\end{lemma}
\begin{proof}
  If $h$ is a projective transformation preserving $\Om$ with only two fixed points in $\bd\Om$ (and
  none inside $\Om$ since we assume $\Gam$ is torsion-free), then $h$ is a biproximal matrix, so
  $h^+$ is an attracting eigenline in $\R^{n+1}$ and $h^-$ is a repelling eigenline. The result
  follows since $h$ preserves $\bd\Om$. 
\end{proof}

\begin{lemma}
  Suppose $h,g$ are hyperbolic projective transformations preserving $\Om$ such that $g^+\neq h^+$.
  Then there exist neighborhoods $V_g,V_h$ of $g^+,h^+$ such that $\mathcal C \cl{V_g}\cap \mathcal
  C \cl{V_h}=\varnothing$ and there is no properly embedded triangle which intersects both $\mathcal
  C \cl{V_g}$ and $\mathcal C \cl{V_h}$. 	
  \label{lem:ballmanstyle}
\end{lemma}

\begin{proof}
  Since $g,h$ are hyperbolic, $g^+,h^+$ are smooth extremal points.  There are disjoint open
  neighborhoods $V_g,V_h$ around $g^+, h^+$ respectively in $\partial \Omega$ whose closures are
  also disjoint, and for which $g^-\not\in \overline{V}_g$ and $h^-\not\in\overline{V}_h$. 
  If the lemma was false, by convexity of $\mathcal C g^nV_g, \mathcal C h^nV_h$, there would exist
  a sequence of properly embedded triangles $\tri_n$ such that $\cl{g^{n}V_g}\cap\bd\tri_n\neq
  \varnothing$ and $ \cl{h^n V_h}\cap \bd\tri_n\neq \varnothing$ for all $n$.  Since the collection
  of properly embedded triangles is closed in $\Om$ \cite[Proposition 3.2]{Ben4}, the $\tri_n$
  accumulate on some $\tri$ properly embedded in $\Om$.  Because $g,h$ are hyperbolic,
  $\cap_{n=1}^\infty\cl{g^n V_g}=\{g^+\}$ and $\cap_{n=1}^\infty \cl{h^n V_h}=\{h^+\}$.  Then $(g^+
  h^+)\subset \cl{\tri}$ which contradicts the smooth extremal property for fixed points of
  hyperbolic isometries. 
\end{proof}

\begin{proof}[Proof of Proposition \ref{prop:geometricprop}]
  Applying Lemma \ref{lem:ballmanstyle}, there are pairwise disjoint neighborhoods $V_h^\pm,V_g^\pm$
  of $h^\pm, g^\pm$ respectively such that no properly embedded triangle intersects any pair of
  convex hulls of these neighborhoods in $\cl{\Om}$. In particular, this means for $V_i,V_j\in
  \{V_h^\pm,V_g^\pm\}$ with $V_i\neq V_j$, for any $x\in \mathcal C V_i$ and $y\in \mathcal C V_j$,
  the projective line $(xy)$ is contained in $\Om$ and is not contained in any single properly
  embedded triangle. 

  By Lemma \ref{fct:hyperbolicisoms}, there exists an $N_1$ such that $h^{-n}(\cl{\Om}\smallsetminus
  \mathcal C V_h^+)\subset V_h^-$ for all $n\geq N_1$.  Moreover, there exists an $N_2$ such that
  $g^{-n} (\cl{\Om}\smallsetminus \mathcal C V_g^+)\subset V_g^-$, implying 
  \[
    g^{-n}(\mathcal C V_h^+)\subset g^{-n} (\cl{\Om}\smallsetminus \mathcal C V_g^+)\subset V_g^-
    \subset \cl{\Om}\smallsetminus \mathcal C V_h^+.
  \]
  Then for all $y\in \Om$ and all $n\geq \max\{N_1,N_2\}$, either $h^{-n}y\in \mathcal C V_h^-$ or
  $h^{-n}g^{-n}y\in \mathcal C V_h^-$. Let $M = \max\{ N_1,N_2\}$. 

  Next, we claim that for all $r>0$, for $\gam\in \{ h^{-M}, h^{-M}g^{-M}\}$ and $R=d_\Om(x,\gam
  x)$, if $\gam y\in \mathcal C V_h^-$ then 
  \[
    \gam^{-1}V_h^+\subset \gam^{-1} \mathcal O_r(\gam y,x) \subset \mathcal O_{r+R}(y,x)
  \]
  which completes the proof of the lemma. Note first that the rightmost inclusion is true for all
  $\gam\in\Gam$:  if $p\in B_\Om(x,r)$, then $d_\Om(\gam^{-1} p,x)\leq d_\Om(\gam^{-1} p,\gam^{-1}
  x)+d_\Om(x,\gam^{-1} x) \leq r + R$. So if a projective  ray $\xi$ with $\xi(0)=\gam y$ intersects
  $B_\Om(x,r)$, then $\gam^{-1} \xi$ is a projective ray with $\gam^{-1}\xi(0)=y$ which intersects
  $B_\Om(x,r+R)$. 

  For the leftmost inclusion, we show that $V_h^+ \subset \mathcal O_r(\gam y,x)$ for sufficiently
  large $r$. First, for any $\eta\in \cl{V_h^+}$ the projective ray $(\gam y\ \eta)$ is contained in
  $\Om$ but not any properly embedded triangle by choice of $V_h^-,V_h^+$ (Lemma
  \ref{fct:hyperbolicisoms}). Then take $r \geq \max_{\eta\in\cl{V_h^+}} d_\Om(x,(\gam y\ \eta))$
  and the leftmost containment is satisfied. 
\end{proof}

\subsubsection{The Shadow Lemma}

\ifextended
\else
First, we need a basic lemma: 
\fi

\begin{lemma} For all $\xi\in\mathcal O_r(x,y)$, 
  \[
    d_\Om(x,y) - 2r \leq \beta_\xi^-(x,y)\leq \beta_\xi^+(x,y) \leq d_\Om(x,y).
  \]
  \label{exer:shadowlem}
\end{lemma}
\begin{proof}
  The rightmost inequality is immediate from the triangle inequality.		For the leftmost
  inequality, let $z\in \Omega$ converge to $\xi$ along the projective line from $x$ to $\xi$.
  Divide the projective line $(xz)$ into two segments by its first intersection $p$ with the closed
  ball $\overline{B_\Om(y,r)}$.  Then by the triangle inequality, $d_\Om(x,y) \leq d_\Om(x,p)+r$ and
  $d_\Om(z,y)\leq d_\Om(z,p)+r$, so 
  \[
    d_\Om(x,y)-2r \leq d_\Om(x,p)+d_\Om(p,z) - d_\Om(y,z) = d_\Om(x,z)-d_\Om(y,z) = \beta_z(x,y).
  \]
  The lower bound follows. 
\end{proof}

\begin{lemma}[Shadow Lemma]
  Let $\mu$ be a 
  Busemann density of dimension $\delta>0$ on $\bdv\Om$. Then
  for every $x\in \Omega$ and all suffiently large $r$, there exists a
  $C>0$ such that for all $\gam\in\Gam$,
  \[
    \frac1C e^{-\delta d_\Om(x,\gam x)}\leq \mu_x\big(\mathcal O_r(x,\gam
    x)\big)\leq C e^{-\delta d_\Om(x,\gam x)}.
  \]
  \label{lem:shadowlem}
\end{lemma}

\begin{proof}
  We follow the elegant proof of Roblin \cite{roblin}.  Since $\gam$ is an
  isometry and by quasi-$\Gamma$-invariance, 
  \begin{align}
    \label{eqn:shadowlem1}
    \mu_x(\mathcal O_r(x,\gam x))  = \mu_x(\gam \mathcal O_r(\gam^{-1} x,x)) = & \ \mu_{\gam^{-1}
  x}(\mathcal O_r(\gam^{-1} x,x)).
\end{align}
By the transformation rule (Definition \ref{def:busemanndensity}),
\begin{multline} \label{eqn:shadowlem2}
  \int_{\mathcal O_r(\gam^{-1} x,x)} e^{-\delta
    \beta^+_\xi(\gam^{-1} x,x)} d\mu_x(\xi)\leq\mu_{\gam^{-1} x}(\mathcal
    O_r(\gam^{-1} x,x)) \\  \leq \int_{\mathcal O_r(\gam^{-1} x,x)}
    e^{-\delta \beta^-_\xi(\gam^{-1} x,x)} d\mu_x(\xi).
  \end{multline}

  Combining Equations \eqref{eqn:shadowlem1} and \eqref{eqn:shadowlem2} with Lemma
  \ref{exer:shadowlem}, 
  \begin{multline}
    \int_{\mathcal O_r(\gam^{-1} x,x)} e^{-\delta d_\Om(\gam^{-1}
  	x,x)}d\mu_x(\xi) \leq
    \mu_x(\mathcal O_r(x,\gamma x)) \\ \leq  \int_{\mathcal O_r(\gam^{-1}
  	x,x)} e^{-\delta
      (d_\Om(\gam^{-1} x,x)-2r)} d\mu_x(\xi),
    \end{multline}
    so, letting $\|\mu_x\|:=\mu_x(\partial\Om)<\infty$,
    \begin{equation}
      e^{-\delta d_\Om(\gam^{-1} x,x)}\mu_x(\mathcal O_r(\gam^{-1} x,x))
      \leq \mu_x(\mathcal O_r(x,\gam x))
      \leq e^{-\delta d_\Om(\gam^{-1} x,x)}e^{2\delta r} \|\mu_x\|.
      \label{eqn:shadowlem}
    \end{equation}
    The rightmost inequality of Equation \eqref{eqn:shadowlem} gives us the rightmost inequality of
    the lemma immediately.  By Proposition \ref{prop:toppaperhyperbolic} there exist two
    noncommuting hyperbolic isometries $g,h$. Then apply Proposition \ref{prop:geometricprop} to
    obtain open sets $O_1=h^MO, O_2=g^Mh^MO \subset\bdv\Om$ such that for all $r$ sufficiently large
    and all $\gam\in\Gam$, either $O_1\subset \mathcal O_r(\gam^{-1}x,x)$ or $O_2\subset \mathcal
    O_r(\gam^{-1}x,x)$.  The $\mu_x$ have full support (Remark \ref{rem:psmeasures}) so we may take
    $0<\frac1C<\min\{\mu_x(O_i)\}$ to complete the proof. 
  \end{proof}

  \subsubsection{Boundaries of flats are null sets}

  Let $S_\tri(x,r):=\{y\in\tri : d_\Om(x,y)=r\}$ denote the sphere of Hilbert radius $r$ about $x$
  restricted to a properly embedded triangle, $\tri$. 
  Similarly, $B_\tri(x,r)$ is the open ball of Hilbert radius $r$ about $x$ restricted to the triangle
  $\tri$. For a properly embedded triangle $\tri$ in $\Om$, let $\Stab_\Gam(\tri)=\{\gam\in \Gam \mid
  \gam\tri=\tri\}$. 

  \begin{lemma}
    Pick a tiling of a properly embedded triangle $\tri$ by
    $\Stab_\Gam(\tri)$ such that $p$ is in the interior of a fundamental
    domain in the tiling.  Choose $R$ so that the open $B_\tri(p,R)$ covers
    the compact fundamental domain containing $p$.  If $N_r$ denotes the
    minimal number of $\gam.B_\tri(p,R)$ which cover $S_\tri(p,r)$, where
    $\gam\in\Stab_\Gam(\tri)$, then  $N_r$ is quasi-linear in $r$. 
    \label{lem:trianglegrowth}
  \end{lemma}

  \begin{proof}
    The projective triangle with the Hilbert metric is isometric to $\R^2$
    with a hexagonal norm \cite{delaHarpe}. By Benoist's Theorem
    \ref{thm:ben3mfldgeom}(c), $\Stab_\Gam(\tri)$ is isomorphic to $\Z^2$
    up to index 2. Under De la Harpe's isometry this $\Z^2$ group acts by
    translations so the growth of orbits of a fundamental domain under the
    hexagonal norm is quasi-linear. 
  \end{proof}

  \begin{proposition}
    The boundary of any properly embedded triangle is a null set for any
    Busemann density of dimension $\delta>0$. 
    \label{prop:nulltriangles}
  \end{proposition}

  \begin{proof}
    Choose a fundamental domain $T$ for the action of $\Stab_\Gam(\tri)$ on
    a properly embedded triangle $\tri$, a point $p\in T$ and $R_0$ as in
    Lemma \ref{lem:trianglegrowth}.  Let $x\in\Omega$ be in a fundamental
    domain $D$ for the $\Gam$-action on $\Om$ such that $T\subset D$.
    Choose $R$ large enough that $D\subset B_\Om(x,R)$ and $B_\tri(p,
    R_0)\subset B_\Om(x,R)$.  Then the $\Gam. B_\Om(x,R)$ covers $\Om$, and
    the minimal number of $\gam$ such that
    $\cup_{i=1}^N\gam_{i}.B_\Om(x,R)$ covers $S_\tri(p,r)$, is bounded
    above by the $N_r$ in Lemma \ref{lem:trianglegrowth}.  For each $r$,
    choose a covering of $S_\tri(x,r)$ by $N_r$-many $\gam_i B_\Om(x,R)$
    and assume that $\gam_i\in\Stab_\Gam(\tri)$ for $i=1,\ldots, N_r$. 

    Next, we show for all large enough $r$, $\bdv\tri\subset
    \bigcup_{i=1}^{N_r} \mathcal O_{2R}(x,\gam_i x) $.  Let $r>2R$.
    Consider any projective ray $\eta$ based at $p$ such that
    $\eta^+\in\bd\tri$.  Let $\xi$ denote projective ray based at $x$ such
    that $\xi^+=\eta^+$.  Then parameterizing $\xi,\eta$ at unit speed, we
    have that $d_\Om(\xi_t,\eta_t)\leq d_\Om(x,p)\leq R$ for all $t\geq0$.
    Since $p\in{\tri}$ and $\eta^+\in \bd\tri$, then $\eta \cap S_\tri
    (p,r)\neq\varnothing$ and there exists a $\gam_i$ such that $\eta\cap
    B_\Om(\gam_ix,R)\neq\varnothing$. Let $t\geq 0$ be such that
    $d_\Om(\eta_t,\gam_ix)<R$. Then $d_\Om(\xi_t,\gam_ix)\leq
    d_\Om(\xi_t,\eta_t)+d_\Om(\eta_t,\gam_ix)\leq 2R$, and $\xi^+ \in
    \mathcal O_{2R}(x,\gam_ix)$.  Lastly, for each $i =1,\ldots, N_r$ let
    $q_i\in S_\tri(p,r)\cap B_\Om(\gam_ix,R)\neq\varnothing$. Then 
    \[
      r=d_\Om(p,q_i)\leq d_\Om(p,x)+d_\Om(x,\gam_ix)+d_\Om(\gam_ix,q_i)
      \leq d_\Om(x,\gam_i x) + 2R
    \]
    implying   $-d_\Om(x,\gam_ix) \leq 2R-r$ for all $i=1,\ldots,N_r$.  By
    Lemma \ref{lem:shadowlem},
    \begin{equation}
      \mu_x(\bdv\tri) \leq \sum_{i=1}^{N_r} \mu_x(\mathcal O_{2R}(x,\gam_i
      x))\leq \sum_{i=1}^{N_r} C e^{-\delta d_\Om(x,\gam_ix)}\leq C
      e^{\delta 2R}e^{-\delta r} N_r.  
      \label{eqn:nulltriangles}
    \end{equation}
    Given that $N_r$ is quasi-linear in $r$ by Lemma
    \ref{lem:trianglegrowth}, that $\delta>0$, and that Equation
    \ref{eqn:nulltriangles} holds for all $r$ sufficiently large, we
    conclude $\mu_x(\partial\tri)=0$.  
  \end{proof}

  Then the following corollary is immediate after Proposition
  \ref{prop:benoist_on_bndy_pts}:

  \begin{corollary}
    The set of smooth extremal points in $\partial\Omega$ is full measure
    for any 
    Busemann density of dimension $\delta>0$.
    \label{cor:smoothfullmeasure}
  \end{corollary}

%

  \section{Busemann densities are unique}		
  \label{sec:psunique}

  In this section we complete the proof of Theorem \ref{thm:mainthm3}.  The
  arguments in this section follow those of Sullivan and Knieper
  \cite{sullivan79, Kn97}. We give brief proofs, mainly to point out when
  we need Corollary \ref{cor:smoothfullmeasure}. 

  \begin{lemma}[Local estimates]
    If $\{\mu_x\}$ is a 
    Busemann density of dimension  $\delta>0$ on
    $\partial\Omega$, then for all $x$ and all sufficiently large $r$ there
    exists a constant $b(r)$ such that for $y\in\Om$ with $d_\Om(x,y)$
    large, 
    \[
      \frac1{b(r)}e^{-\delta d_\Om(x,y)} \leq \mu_x(\mathcal O_r(x,y))\leq
      b(r) e^{-\delta d_\Om(x,y)}
    \]
    \label{lem:localestimate}
  \end{lemma}


  \begin{proof}
    Note that if $y=\gam x$, then we apply Lemma \ref{lem:shadowlem} to
    obtain the result. Else, for some $\Gam$-tiling of $\Om$ with compact
    fundamental domain $D$, choose $r$ large enough that for all $x\in D$,
    we have $D\subset B_\Om(x,\frac{r}2)$.  Choosing $D$ such that $y\in
    D$, there exists a $\gam\in\Gam$ such that $\gam x\in D\subset
    B_\Om(y,\frac{r}2)$.  By the triangle inequality, 
    \[
      \mathcal O_{\frac{r}2}(x,\gam x) \subset \mathcal O_r(x,y) \subset
      \mathcal O_{\frac{3r}2}(x,\gam x).
    \]

    Applying Lemma \ref{lem:shadowlem}, if $r$ is sufficiently large then
    there is a uniform constant $C$ such that 
    \begin{align*}
      \frac1C e^{-\delta d_\Om(x,\gam x)}\leq \mu_x(\mathcal
      O_{\frac{r}2}(x,\gam x)) & \leq
      \mu_x(\mathcal O_r(x,y))  \\ & \leq \mu_x(\mathcal
      O_{\frac{3r}2}(x,\gam x)) \leq C e^{-\delta d_\Om(x, \gam x)}. 
    \end{align*}
    Our final observation is that since $\gam x\in B_\Om(y,r/2)$, 
    \[
      \frac1{Ce^{\delta r/2}}e^{-\delta d_\Om(x,y)} \leq \mu_x(\mathcal
      O_r(x,y))\leq C e^{\delta r/2}e^{-\delta d_\Om(x,y)}.
    \]
  \end{proof}

  It follows that 
  Busemann densities have no atomic part:

  \begin{corollary}
    Busemann densities of dimension $\delta>0$ on $\partial\Omega$ have no
    atoms. 
    \label{cor:lowerbound}
  \end{corollary}

  \begin{proof}	
    It suffices to check for smooth extremal points $\xi\in\partial \Omega$
    by Corollary \ref{cor:smoothfullmeasure}.  Let $y_n$ be a sequence of
    points in $\Omega$ converging to $\xi$ along a projective line. Then
    apply the local estimate lemma (Lemma \ref{lem:localestimate}), for
    fixed sufficiently large $R$, to the shadows $\mathcal O_x(y_n,R)$. The
    conclusion follows Lemma \ref{lem:shadows_converge_to_smoothext}. 
  \end{proof}

%


  \begin{corollary}
    Busemann densities of dimension $\delta>0$ are equivalent. 
    \label{cor:equivalentdensities}
  \end{corollary}

  \begin{proof}
    Let $\{\mu_x\},\{\nu_x\}$ be Busemann densities of dimension $\delta$.
    Let $\xi\in\partial\Om$ be a smooth extremal point and take a sequence
    $y_n$ of points in $\Omega$ converging to $\xi$ along a projective
    line. Then for all sufficiently large $n$, $d_\Om(x,y_n)$ is large
    enough to apply Lemma \ref{lem:localestimate} to both densities 
    and $\nu$ and conclude:
    \[
      \frac{1}{b_\mu(r)b_\nu(r)}\leq \frac{\mu_x(\mathcal
      O_r(x,y_n))}{\nu_x(\mathcal O_r(x,y_n))} \leq b_\mu(r)b_\nu(r).
    \] 
    By Lemma \ref{lem:shadows_converge_to_smoothext}, 
    since $y_n$ converges to $\xi$ along a
    projective line and $\xi$ is smooth and extremal, the shadows
    $\mathcal O_r(x,y_n)$ form a nested decreasing sequence with
    intersection $\{\xi\}$.  Thus, since smooth extremal points form a set
    of full measure for any $\delta$-dimensional Busemann density by
    Corollary \ref{cor:smoothfullmeasure}, we conclude that $\mu_x$ and
    $\nu_x$ are equivalent. 
  \end{proof}

  \begin{proposition}
    If $\{\mu_x\}$ is a 
    Busemann density of dimension $\delta>0$ on
    $\partial\Omega$, then for all $x\in\Omega$, the measure $\mu_x$ is ergodic
    for the $\Gam$-action on $\partial\Omega$.
    \label{prop:ergodicdensities}
  \end{proposition}

  \begin{proof}
    Let $A\subset \partial\Omega$ be a Borel, $\Gam$-invariant set with
    positive $\mu_x$-measure for all $x$, since the measures are equivalent.
    Define a new density $\bar{\mu}_x(B):=\mu_x(A\cap B)$ for all $x\in \Om$.
    Since $A$ is $\Gam$-invariant and has positive measure, it suffices to
    show that $\bar{\mu}_x$ is a Busemann density also of dimension $\delta$.
    Then $\mu_x$ is equivalent to $\bar{\mu}_x$ by Corollary
    \ref{cor:equivalentdensities}, and we conclude that $\mu_x
    (\partial\Om\smallsetminus A) = \bar{\mu}_x(\partial\Om\smallsetminus A)
    = 0$, proving ergodicity of $\mu_x$ for $\Gam$.

    It is clear that $\bar{\mu}_x$ is nontrivial and finite.  Since smooth
    extremal points are full measure and the transformation rule is
    well-defined for smooth extremal points, the proof that $\bar{\mu}_x$
    satisfies the transformation rule does not differ significantly from 
    \cite[Proposition 4.15]{Kn97}, and the proof of quasi-$\Gam$-invariance
    is unchanged. 
  \end{proof}


  \begin{theorem}
    Busemann densities of dimension $\delta>0$ on
    $\partial\Omega$ are unique up to a constant. 
    \label{thm:uniquedensities}
  \end{theorem}
  \begin{proof}
    Let $\{\mu_x\}, \{\nu_x\}$ be two Busemann densities of dimension
    $\delta$.  Since $\mu_x$ and $\nu_x$ are equivalent, it suffices to
    show that the Radon-Nikodym derivative $d\nu_x/d\mu_x$ is
    $\Gam$-invariant on the set of smooth extremal points, which are a set
    of full measure by Corollary \ref{cor:smoothfullmeasure}.  Ergodicity
    of $\mu_x$ then implies that the Radon-Nikodym derivative is constant
    $\mu_x$-almost everywhere. Since the densities have the same
    transformation rule almost everywhere, this constant does not depend on
    $x$. 
    Verifying that the Radon-Nikodym derivative is $\Gam$-invariant on the
    set of smooth extremal points is straightforward. 
  \end{proof}

  Combining 
  Proposition \ref{prop:existenceconfdensity} 
  and Theorem \ref{thm:uniquedensities} gives us
  Theorem \ref{thm:mainthm3}. 

  \subsection{Volume growth and divergence of $\Gam$}

  In this section, we see that $\Gam$ is divergent.  With all the tools is
  place, the proof does not differ from that of Knieper for rank one
  manifolds, but we include it here for completeness \cite[Theorem
  5.1]{Kn97}.

  First we prove Theorem \ref{thm:mainthm4} on the growth rate of volumes
  of spheres.  Let $\vol_{x,t}$ be a Hilbert volume
  form on $S_\Omega(x,t)$ which is $\Gamma$-equivariant, meaning
  $\vol_{\gamma x,t}=\gamma_\ast \vol_{x,t}$. We abbreviate with 
  $\vol$ when the context is
  clear.  For a definition of sphere volume, see the definiton of area of a
  smooth hypersurface in \cite[Equation 1.3]{vernicos13}. Since the
  nonsmooth points in the spheres come from nonsmooth points in the
  boundary of $\Omega$, which form a countable set in this case by Benoist
  (Theorem \ref{thm:ben3mfldgeom}), these nonsmooth points in spheres are
  measure zero and the definition can be applied. Indeed, Vernicos studies
  asymptotics of this sphere volume for any Hilbert geometry, not
  necessarily smooth ones (see \cite[Theorem 2.1]{vernicos13}).

  \begin{proof}[Proof of Theorem \ref{thm:mainthm4}]
    Let $\{\mu_x\}$ denote the Patterson-Sullivan density of dimension
    $\delta_\Gam$, existence of which we constructed in
    Proposition \ref{prop:existenceconfdensity}. 
    By previous work we have that $\delta_\Gam>0$ \cite{braytop}. 
    Let $\delta=\delta_\Gam$ throughout the proof. 
    Let $R$ be sufficiently large to
    apply Sullivan's Shadow Lemma (Lemma \ref{lem:shadowlem}) and
    consequently the local estimate in Lemma \ref{lem:localestimate}.
    Consider $r\geq 6R$. By compactness of $S_\Om(x,t)$, we can take
    $\{x_i\}_{i=1}^{N_t}$ to be a maximal $r$-separating set
    in $S_\Om(x,t)$. In particular, if $i\neq j$, then $d_\Om(x_i,x_j)\geq
    r$, implying $B_\Om(x_i,r/3)\cap B_\Om(x_j,r/3)=\varnothing$.
    Maximality implies $S_\Om(x,t)\subset \bigcup_{i=1}^{N_t}B_\Om(x_i,r)$. 

    By the local estimate (Lemma \ref{lem:localestimate}), there exists a
    $b(x)$ such that for all $r\in[2R,6R]$ and $x_i$, each of which is
    distance $t$ from $x$, we have 
    \[
      \frac1be^{-\delta t}\leq \mu_x(\mathcal O_r(x,x_i))\leq be^{-\delta t}.
    \]
    Then
    \[
      \mu_x(\bdv\Om)\leq \mu_x(\bigcup_{i=1}^{N_t}\mathcal O_r(x,x_i)) \leq
      \sum_{i=1}^{N_t}\mu_x(\mathcal O_r(x,x_i))\leq N_t b e^{-\delta t},
    \]
    and 
    \[
      \mu_x(\bdv\Om)\geq \mu_x(\bigcup_{i=1}^{N_t}\mathcal O_{r/3}(x,x_i)) =
      \sum_{i=1}^{N_t}\mu_x(\mathcal O_{r/3}(x,x_i)) \geq \frac{N_t}b e^{-\delta
      t}.
    \]
%
    Since $\mu_x(\partial\Omega)$ is a constant depending on $x$, there is a number
    $b'(x)$ such that 
    \[
      \frac1{b'}e^{\delta t}\leq N_t \leq b' e^{\delta t}.
    \]

    By cocompactness of $\Gamma$ acting on $\Omega$, and
    $\Gam$-equivariance of the sphere volumes $\vol_{x,t}$ on
    the spheres of radius $t$ around $x$, 
    there exists an $\ell(r)$ such that for all $y\in S_\Omega(x,t)$, 
    \[
      \frac1\ell \leq \vol_{x,t}(B_\Om(y,r)\cap S_\Om(x,t))\leq \ell .
    \]
    Then we may arrange for an $\ell'$ such that 
    \[
      \vol_{x,t}(S_\Om(x,t))\leq
      \sum_{i=1}^{N_t}\vol_{x,t}(B_\Om(x_i,r)\cap S_\Om(x,t)) \leq
      N_t \cdot \ell' \leq \ell'b'e^{\delta t} 
    \]
    and 
    \[
      \vol_{x,t}(S_\Om(x,t))\geq \sum_{i=1}^{N_t}
      \vol_{x,t}(B_\Om(x_i, r/3)\cap S_\Om(x,t))
      \geq \frac{N_t}{\ell'}\geq \frac1{\ell'b'} e^{\delta t}, 
    \]
    which concludes the proof of the theorem. 
  \end{proof}

%
%

  \begin{corollary}
    Let $\Om$ be a properly convex, divisible, indecomposable Hilbert geometry of dimension three with
    dividing group $\Gam$.  Then $\Gam$ is of divergent type. 
    \label{cor:divergent}
  \end{corollary}

  \begin{proof}
    Let $D$ be a compact fundamental domain for the $\Gam$-action on $\Om$. Then for $s>\delta_\Gam$,
    $P(x,y,s)$ converges so we can apply Fubini's Theorem to the following integral:
    \begin{multline*}
      \int_D \sum_{\gam\in\Gam}e^{-sd_\Om(x,\gam y)}d\vol(y) = \sum_{\gam\in\Gam} \int_{\gam(D)}e^{-s
      d_\Om(x,\gam y)} d\vol(y) \\= \int_0^\infty e^{-s t}\vol(S_\Om(x,t))dt.
    \end{multline*}
    As $s$ decreases to $\delta_\Gam$, the right hand side diverges 
    \ifextended
\else
    by Theorem \ref{thm:mainthm4}.
    \fi
  \end{proof}

  \section{The Bowen-Margulis measure}	
  \label{sec:bowenmargulismeasure}

  In this section, we introduce the $\Gam$-invariant Bowen-Margulis measure
  on $T^1\Om$, denoted $\wt\mu_{BM}$, following the standard construction
  \cite{sullivan79, roblin,Kn97} and prove Theorem \ref{thm:mainthm1}.

  \subsection{Definition and properties}
  Let $\{\mu_x\}$ be the Patterson-Sullivan density constructed in
  Proposition \ref{prop:existenceconfdensity}, which is a 
  Busemann density of dimension $\delta_\Gamma>0$. 
  For each $x\in \Om$ and Borel set
  $A\subset T^1\Om$, define
  \begin{equation} \label{eqn:defnbmm}
    \wt\mu_{BM}^x(A) 
    = \int_{v\in A} \length_\Om(\ell_v\cap \pi A)
    e^{\delta_\Gam(\beta_{v^+}(x,\pi v)+\beta_{v^-}(x,\pi v))} \
    d\mu_x(v^-)d\mu_x(v^+)
  \end{equation}
  where $\pi\colon T^1\Om\to\Om$ is the footpoint projection and
  $\length_\Om$ is Hilbert length. 
  Since the Busemann function is well-defined almost everywhere (Lemma
  \ref{lem:busemann} and Corollary \ref{cor:smoothfullmeasure}), this
  definition is valid.  Then $\wt\mu_{BM}^x$ is $\Gam$-invariant by the
  definition of a $\delta_\Gam$-dimensional Busemann density and the
  cocycle property of the Busemann function.  On $T^1M$ the measure is
  finite, and we may normalize it so $\mu_{BM}^x(T^1M)=1$. 

  Recall that $T^1\Omega_{\reg}$ is the set of regular vectors, which are
  vectors $v$ whose endpoints $v^-,v^+$ in $\partial\Omega$ are both smooth
  and extremal, and $T^1M_{\reg}$ is the projection of these vectors to
  $T^1M$ (Definition \ref{def:regularvectors}).

  The following Lemma is clear given the discussion above. 
  \begin{lemma}
    The regular set $T^1M_{\reg}$ is a set of full $\mu_{BM}$-measure. 
    \label{lem:regularsetfullmeasure}
  \end{lemma}

  For the remainder of the paper, we will let $\mu_{BM}:=\mu_{BM}^x$ for
  some $x\in\Om$.  Note that since the $\mu_{BM}^x,\mu_{BM}^y$ are
  equivalent by construction, we will have that they are in fact equal up
  to a constant after the proof of ergodicity is complete. 

  \subsection{Ergodicity}
  \label{sec:ergodic}

  Let $d$ be the Finsler metric on $T^1M$ discussed in \cite[Section
  4.1]{braytop}.  Define the $\phi^t$-invariant strong unstable foliations
  for $v\in T^1M_{\reg}$ to be
  \[
    W^{su}(v)=\{w\in T^1M \mid d(\phi^{-t}v,\phi^{-t}w)\to0 \text{ as
    }t\to+\infty\}
  \]
  and similarly for $W^{ss}(v)$, the strong stable foliation, which is
  contracted in forward time.  The weak unstable set $W^{ou}(v)$ is the
  disjoint union of $W^{ss}(\phi^tv)$ for all $t\in\R$, and similarly for
  the weak stable set $W^{os}(v)$. This gives us a flow-invariant foliation
  of the weak unstable sets by strong unstable leaves, and similarly for
  the stable foliation.

  \begin{lemma}
    For all regular $v,w$ with $v\neq -w$, we have 
    \[
      W^{ss}(v)\cap W^{ou}(w)\neq\varnothing. 
    \]
    \label{lem:global leaves}
  \end{lemma}
  \begin{proof}
    If $v$ is regular then $W^{ss}(v)$ is defined by a geometric
    characterization on the universal cover \cite{Ben1,braytop}: 
    \begin{align*}
      \wt W^{ss}(v) & = \{w\in T^1\Om \mid v^+=w^+, \ \pi w\in \mathcal
	H_{v^+}(\pi v)\},  \\
      \wt W^{os}(v) & = \{w\in T^1\Om \mid v^+=w^+\}.
    \end{align*}
    where $\mathcal H_{v^+}(\pi v)$ is the globally defined horosphere
    through $\pi v$ at $v^+$ (Corollary \ref{cor:horospheres}).  
    \ifextended
\else
    The result follows the geometric interpretation of the Busemann
    function in Lemma \ref{lem:busemann}.
    \fi
  \end{proof}

  \subsubsection{The Hopf argument}

  We first establish or recall basic facts which set up the ergodicity proof. 
  Let $f\colon T^1M\to\R$ be integrable. Then the forward and backward
  Birkhoff averages of $f$ for $\phi$ are , respectively, 
  \begin{align*}
    &f^+(v)=\lim_{t\to+\infty}\frac1t\int_0^t f\circ \phi^s(v)\ ds, \\
    &f^-(v)=\lim_{t\to+\infty}\frac1t\int_0^t f\circ \phi^{-s}(v)\ ds. 
  \end{align*}
  By the Birkhoff ergodic theorem, $f^+$ and $f^-$ exist for
  $\mu_{BM}$-almost every $v\in T^1M$ (see \cite[Theorem 4.1.2]{MDS}). 
  The following lemma is straightforward to verify by compactness of $T^1M$:

  \begin{lemma}
    Forward Birkhoff averages of continuous functions are constant on strong stable leaves of
    regular vectors and backward Birkhoff averages as constant on strong unstable leaves.
    \label{lem:constant on leaves}
  \end{lemma}

  Since $\phi$ is invertible, $f^+=f^-$ $\mu_{BM}$-almost
  everywhere
  (see \cite[Proposition 4.1.3]{MDS}). 
  \ifextended
  \else
  We have the following classical lemma, which we do not prove here, which allows us to verify
  ergodicity by proving $f^+$ is constant almost everywhere for all continuous $f$. 
  \fi

  \begin{lemma}
    If $f^+$ is constant $\mu_{BM}$-almost everywhere for all continuous $f$, then every
    $\phi^t$-invariant $L^1$-integrable function is constant $\mu_{BM}$-almost everywhere. 
    \label{lem:hopf cts check}
  \end{lemma}

  We make the arguments locally in the universal cover and conclude ergodicity by transitivity of the
  flow on the quotient.
  We define strong unstable conditional measures as induced Patterson-Sullivan measure on strong
  unstable leaves:
  \begin{equation}
    \mu_v^{su}(A) = \int_{w\in A\cap W^{su}(v)}e^{\delta_\Gam \beta_{w^+}(x,\pi w)} \
    d\mu_x(w^+).
    \label{def:unstableconditionals}
  \end{equation}
  We can define the strong stable conditionals $\mu_v^{ss}$ on $W^{ss}(v)$ similarly. 
  Note that for $w\in W^{su}(v)$, we have $\pi w\in \mathcal H_{v^-}(\pi v)$ and $w^-=v^-$, hence
  $\beta_{w^+}(x,\pi w)$ is constant over $w\in W^{su}(v)$ and the conditional measures will
  not depend the point in a leaf of the foliation. 
  We will say the strong unstable foliation is \emph{absolutely continuous}
  if the associated strong stable conditionals are absolutely continuous as
  measures.

  \begin{lemma}
    The strong unstable foliations are absolutely continuous for all
    regular points in $T^1M$. 
    \label{lem:abscts}
  \end{lemma}
  \begin{proof}
    For each $t$ we have uniform contraction along flow lines:
    \begin{equation}
      \frac{d\phi_\ast^t\mu_{v}^{su}}{d\mu_{\phi^tv}^{su}} = e^{-\delta_\Gam t}
      \label{eqn:uniformcontraction}
    \end{equation}
    by the cocycle property of the Busemann function. 
    This gives us absolute
    continuity of the strong
    unstable conditionals along flow lines.

    It remains to consider $v,w$ regular vectors on the same strong stable leaf. 
    To determine absolute continuity of the strong unstable conditionals,
    we define a measurable bijection $h\colon W^{su}(v)\to W^{su}(w)$
    where, for $u\in W^{su}(v)$, we let $h(u)$ be the unique regular vector
    such that $h(u)^-=w^-$, $h(u)^+=u^+$, and $\beta_{w^-}(\pi w,\pi
    h(u))=0$; in other words, $h(u)\in W^{su}(w)\cap W^{os}(u)$ (Lemma
    \ref{lem:global leaves}).  Then we compute the density 
    \[
      \rho_{v,w}(u) := \frac{d  \mu_w^{su}}{dh_\ast\mu_v^{su}}(u) =
      e^{-\delta_\Gam \beta_{u^+}(\pi u, \pi h(u))}
  \]
  and see that $0<\rho_{v,w}(u) <\infty$ as well. To complete the
  computation, let $A\subset W^{su}(w)$. Then 
  \begin{align*}
    \int_{A} \rho_{v,w}(u) \ d h_\ast \mu_v^{su}(u) & = 
    \int_{u\in A} e^{-\delta_\Gam \beta_{u^+}(\pi u, \pi h(u))} e^{\delta_\Gam
      \beta_{u^+}(x,\pi h(u))} \ d\mu_x(u^+)	\\
      & = \int_{u\in A} e^{\delta_\Gam \beta_{u^+}(x,\pi u)}\ d\mu_x(u^+) =
      \mu_w^{su}(A) 
    \end{align*}
    by the cocycle property of the Busemann function. 
  \end{proof}

  \begin{remark}
    The final remark we make before proving ergodicity is that, locally, the
    $\mu_{BM}$-measure of a Borel set $A$ agrees with the
    $\tl{\mu}_{BM}$-measure of a lift of $A$, so we can exploit the
    Patterson-Sullivan product structure of $\mu_{BM}$ on such sufficiently
    small neighborhoods (see Definition \ref{eqn:defnbmm}). We will refer to
    this feature as the {\em local product structure of} $\mu_{BM}$. Then it
    is clear that, for such a small Borel measurable set $N\subset T^1M$
    which we identify with a lift in $T^1\Om$, we have $\mu_{BM}(N)=0$ if and
    only if $\mu_v^{su}(N)=0$ for $\mu_{BM}$-almost every $v$. 
    \label{rem:localproductstructure}
  \end{remark}
  In the arguments below, we abuse notation and treat the measures as conditional measures on
  a small neighborhood in the universal cover. 

  \begin{theorem}
    The Bowen-Margulis measure is ergodic for the geodesic flow.
    \label{thm:ergodic}
  \end{theorem}
  \begin{proof}
    Let $f$ be a continuous function and 
    $\Lambda_q = \{v\in T^1M \mid f^+(v)\geq q\}$ for some $q\in \mathbb Q$ such that
    $\mu_{BM}(\Lambda_q)>0$. Then $\mu_v^{su}(\Lambda_q)>0$ for $\mu_{BM}$-almost every $v\in
    T^1M_{\reg}$ by the local product structure of the Bowen-Margulis measure and that
    $T^1M_{\reg}$ has full $\mu_{BM}$-measure
    (Lemma \ref{lem:regularsetfullmeasure}). 
    By Lemma \ref{lem:abscts}, the unstable conditionals are absolutely continuous
    for every pair of regular vectors, so $\mu_v^{su}(\Lambda_q)>0$ for
    \emph{every} regular vector $v$. 
    Let $G$ be the set of full $\mu_{BM}$-measure on which $f^-=f^+$ (by invertibility of the
    flow and \cite[Proposition 4.1.3]{MDS}). Then $G$ is also a set of
    full $\mu_v^{su}$-measure for $\mu_{BM}$-almost every $v$. Then for almost every $v$, we
    have $\mu_v^{su}(\Lambda_q)>0$ which implies $\mu_v^{su}(\Lambda_q\cap G)>0$, and so there exists a
    $w\in \Lambda_q \cap G\cap W^{su}(v)$. Thus for all $u\in G\cap W^{su}(v)$, a full
    $\mu_v^{su}$-measure set, we have \[
      f^+(u)=f^{-}(u) = f^-(w) = f^+(w) \geq q
    \]
    since $f^-$ is constant on strong unstable sets by $\phi^t$-invariance of $f^-$ (Lemma
    \ref{lem:constant on leaves}).
    Thus, $u\in \Lambda_q$ and $\mu_v^{su}(\Lambda_q)=1$ for $\mu_{BM}$-almost every $v$.
    This implies $\mu_{BM}(\Lambda_q)=1$ by the local product structure of $\mu_{BM}$. 
    Since $\mathbb Q$ is dense in $\mathbb R$ we conclude $f^+$ is constant on a set of full
    measure and by Lemma \ref{lem:hopf cts check} the proof is complete. 
      \end{proof}

      \subsection{A measure of maximal entropy}
      \label{sec:mme}

      The measure-theoretical entropy of $\mu$ with respect to the finite
      measureable partition $\mathcal A =\{A_1,\ldots,A_m\}$, also know as
      the Kolmogorov-Sinai entropy, is 
      \[
	h_\mu(\phi^1,\mathcal A) = \lim_{n\to\infty} \frac1n H_\mu(\mathcal
	A_\phi^{(n)})
      \]
      where $H_\mu(\mathcal B)=-\sum_{i=1}^k \log(\mu(B_i)) \mu(B_i)$ is
      the entropy of a finite measurable partition $\mathcal
      B=\{B_1,\ldots,B_k\}$ and $\mathcal A_\phi^{(n)}:=\bigvee_{i=0}^{n-1}
      \phi^{-i}\mathcal A$ is the partition consisting of all intersections
      $\bigcap_{i=0}^{n-1} \phi^{-i} A_{j_i}$ over all possible
      $\{j_1,\ldots,j_{n-1}\}\subset\{1,\ldots,m\}$. Then the
      measure-theoretic entropy of the pair $(\phi^1,\mu)$ is 
      \[
	h_\mu(\phi^1) = \sup_{\mathcal A} h_\mu(\phi^1,\mathcal A)
      \]
      and the entropy of $\mu$ for the geodesic flow $\phi^t$ is
      $h_\mu:=h_\mu(\phi^1)$.  By work in \cite{braytop}, $\phi^t$ is
      entropy-expansive with expansivity constant $\ep>0$.  Then by
      \cite[Theorem 3.5]{bowen72}, $h_\mu = h_\mu(\phi,\mathcal A)$ for
      $\diam(\mathcal A)<\ep$. 

      \begin{lemma}
	There exists some $a>0$ such that 
	\[
	  \mu_{BM}(\alp) \leq e^{-\delta_\Gam n}a
	\]
	for all $\alp\in\mathcal A_\phi^{(n)}$. 
	\label{lem:mme}
      \end{lemma}
      \begin{proof}
	\ifextended
\else
	The proof does not differ from \cite[Lemma 2.5]{Kn98}.
	\fi
      \end{proof}

      \begin{theorem}
	The Bowen-Margulis measure is a measure of maximal entropy.
	\label{lem:maxentropy}
      \end{theorem}
      \begin{proof}
	\ifextended
\else 
	The proof is as in \cite[Theorem 5.12]{Kn98}.  First, using Lemma \ref{lem:mme} one computes
	$H_{\mu_{BM}}(\mathcal A_\phi^{(n)}) \geq \delta_\Gam n -\log a$. 
	\fi
	In \cite[Proposition 7.3]{braytop}, since the quotient is compact
	and by a technical lemma of Crampon for Hilbert geometries
	\cite[Lemma 8.3]{Cr09}, the arguments of Manning extend
	\cite{manning} allowing us to conclude 
	$\delta_\Gam = h_{top}$. 
	Then by the variational principle,  
	\begin{align*}
	  h_{top} \geq h_{\mu_{BM}} & = \lim_{n\to\infty} \frac1n
	  H_{\mu_{BM}}(\mathcal A_\phi^{(n)}) \geq \lim_{n\to\infty}
	  \frac1n \left( \delta_\Gam n -\log a \right) = \delta_\Gam =
	  h_{top}. 
	\end{align*}
      \end{proof}

      \bibliographystyle{alpha} 
      \bibliography{second_revisions_biblio}

      \bigskip
      \bigskip
      \noindent \textit{Department of Mathematics, University of Michigan, Ann Arbor 48109}

      \noindent 
      \href{mailto:hbray@umich.edu}{\texttt{hbray@umich.edu}}

      \end{document}